\documentclass{amsart}
\usepackage{amsmath}
\usepackage{color}
\usepackage{mathtools}

\newtheorem{theorem}{Theorem}[section]
\newtheorem{lemma}[theorem]{Lemma}
\newtheorem{proposition}[theorem]{Proposition}
\newtheorem{corollary}[theorem]{Corollary}

 \theoremstyle{definition}
\newtheorem{definition}[theorem]{Definition}
\newtheorem{example}[theorem]{Example}

\theoremstyle{remark}

\numberwithin{equation}{section}

\begin{document}

\title[Fully non-linear equations on noncompact manifolds]
{Fully non-linear elliptic equations on noncompact complex manifolds}

\author{Hanzhang Yin}
\address{School of Mathematics, Harbin Institute of Technology,
         Harbin, Heilongjiang 150001, China}
\email{YinHZ@hit.edu.cn}

\begin{abstract}

In this paper, we establish a priori estimates and existence results for solutions of a general class of
fully non-linear equations on noncompact K\"{a}hler and Hermitian manifolds. As geometric applications, we construct complete K\"{a}hler metrics with prescribed volume forms on strictly pseudoconvex domains, as well as find Einstein metrics on complete noncompact K\"{a}hler manifolds and Hessian manifolds with negative first Chern class.


\end{abstract}

\maketitle

\section{Introduction}

Let $(M,\alpha)$ be a complete Hermitian manifold of dimension $n$. Fix a real $(1,1)$-form $\chi$ on $(M,\alpha)$, for any $C^2$ function $u:M\rightarrow \mathbb{R}$ we have a new real $(1,1)$-form $g=\chi+\sqrt{-1}\partial\bar{\partial}u$, and we can define the endomorphism of $T^{1,0}M$ given by $A_j^i=\alpha^{i\bar{p}}g_{j\bar{p}}$. In this paper, we consider the equation for $u$ as follows:
\begin{equation}\label{e1.1}
F(A)=h(x,u),
\end{equation}
for a given function $h$ on $M$, where
\begin{equation}
F(A)=f(\lambda_1,\ldots,\lambda_n)
\end{equation}
is a smooth symmetric function of the eigenvalues of $A$. Such equations have been studied by Sz\'{e}kelyhidi \cite{S18} on compact Hermitian manifolds.

There are several structure conditions for equation \eqref{e1.1}:

\textbf{(a1)} $f$ is defined on an open symmetric convex cone
$\Gamma \subseteq \mathbb{R}^n$ and $\Gamma\neq \mathbb{R}^n$, containing the positive orthant $\Gamma_n=\{(x_1,\ldots,x_n)\in \mathbb{R}^n:x_i>0,i=1,\ldots,n\}$.

\textbf{(a2)} $f$ is symmetric, smooth, concave and increasing, i.e. its partials satisfy $f_i>0$ for all $i$.

\textbf{(a3)} If $h(x)$ is independent of $u$, we only require: $\sup_{\partial \Gamma}f<\inf_M h$;
\\\hspace*{1.25cm}If $h(x,u)$ depends on both $x$ and $u$, we instead require: $\sup_{\partial \Gamma}f=-\infty$.

\textbf{(a4)} For all $\mu\in \Gamma$, we have
\begin{equation}
\lim_{t\rightarrow\infty}f(t\mu)=\sup_{\Gamma} f,
\end{equation}
where both sides are allowed to be $\infty$.

\textbf{(a5)} $h_u(x,u)\geq 0$ and $\sup_M h(x,0)<\sup_\Gamma f$.

Here are some familiar examples of such non-linear operators $f(\lambda)$: When we take $f(\lambda_1,\ldots,\lambda_n)=\log\lambda_1\ldots\lambda_n$,  $\Gamma=\Gamma_n$, \eqref{e1.1} is the complex Monge-Amp\`{e}re equation, Yau \cite{Yau78} first solved such equations on compact K\"{a}hler manifolds. Guo-Phong \cite{GP23} established sharp $L^\infty$-estimates of complex Monge-Amp\`{e}re equations by PDE methods and without pluripotential theory. A related equation, the $(n-1)$ Monge-Amp\`{e}re equation was introduced by Fu-Wang-Wu \cite{FWW10}, in this case we take $f=\log \widetilde{\lambda_1}\ldots\widetilde{\lambda_n}$, $\Gamma=\Gamma_n$, where
\[\widetilde{\lambda_k}=\frac{1}{n-1}\sum_{i\neq k}\lambda_i,\]
for each $k$. Note that a class of $(n-1)$ Monge-Amp\`{e}re equations falls within our framework (see Section 7); further studies on such equations can be found in \cite{FWW15,JL23,TW17,TW19}.

The more general examples are given by $f=\sigma_k^{\frac{1}{k}}$ and $f=(\sigma_k/\sigma_l)^{\frac{1}{k-l}}$, $1\leq l<k\leq n$ defined on the cone
\[\Gamma_k=\{\lambda\in \mathbb{R}^n:\sigma_j(\lambda)>0,j=1,\ldots,k\},\]
where $\sigma_k$ is the $k$-th elementary symmetric  polynomial
\[\sigma_k(\lambda)=\sum_{i_1<\cdots<i_k}\lambda_{i_1}\cdots\lambda_{i_k},\;\;\;1\leq k\leq n.\]
Consider $f=(\sigma_n/\sigma_l)^{\frac{1}{n-l}}(\lambda)=h$, when $l=n-1$ and $h$ is constant, Song-Weinkove \cite{SW08} showed that a solution exists on K\"{a}hler manifolds if there is a $\mathcal{C}$-subsolution (see Definition \ref{d4.1}) for \eqref{e1.1}. Fang-Lai-Ma \cite{FLM11} generalized this result to general $k$.

Since we only require $h_u(x,u)\geq 0$ in condition \textbf{(a5)}, the $C^{0}$-estimates for \eqref{e1.1} is a difficulty of the paper. We begin by reviewing some related works. Yau \cite{Yau78} established the $C^{0}$-estimates for Monge-Amp\`{e}re equations on compact K\"{a}hler manifolds and solved the Calabi conjecture. Tosatti-Weinkove \cite{TW10} used Morse iteration to derive the $L^\infty$-estimates for Monge-Amp\`{e}re equations on compact Hermitian manifolds, generalizing the Calabi conjecture to Hermitian manifolds. Guo-Phong \cite{GP24} established the $L^\infty$-estimates for general classes of fully non-linear equations on Hermitian manifolds, their method can also be applied to open manifolds with a positive lower bound on their injectivity radii, the estimates on open manifolds rely on the behavior of the solution function $u$ in the interior of the open manifolds.

In this paper, we establish the $C^{0}$-estimates for \eqref{e1.1} that depend only on the behavior of the solution function $u$ at infinity on the open manifolds. As we know, any K\"{a}hler form $\omega$ can be locally expressed by $\omega=\sqrt{-1}\partial\bar{\partial}\varphi>0$, where $\varphi$ is called the K\"{a}hler potential. The key idea of $C^{0}$-estimates is to assume that the complex manifold $M$ admits a bounded global K\"{a}hler potential, see the following definition:

\begin{definition}\label{d1.1}
A complex manifold $M$ is said to have a bounded global K\"{a}hler potential if and only if there exists a bounded $C^\infty$ function $\varphi:M\rightarrow \mathbb{R}$ such that $\sqrt{-1}\partial\bar{\partial}\varphi>0$ (by subtracting a constant, we may assume that $\varphi<0$).
\end{definition}

\textbf{Remark}: The above assumption implies that $M$ is a noncompact K\"{a}hler manifold, because $M$ admits a K\"{a}hler form $\sqrt{-1}\partial\bar{\partial}\varphi>0$, and if $M$ is compact, the maximal principle shows that $\varphi$ is a constant, this contradicts the fact that $\sqrt{-1}\partial\bar{\partial}\varphi>0$. And not all noncompact complex manifolds satisfy the assumption. For example, $\mathbb{C}^n$ does not, because if $\sqrt{-1}\partial\bar{\partial}\varphi>0$, this means $4\varphi_{i\bar{i}}=\frac{\partial^2\varphi}{\partial x_i^2}+\frac{\partial^2\varphi}{\partial y_i^2}>0$, and any subharmonic function on $\mathbb{R}^2$ that is bounded must be constant (see \cite{AGR01}, Chapter 11). In section 3, we will give some examples that satisfy the assumption, and construct some metrics with bounded geometry on such manifolds.

The $C^{0}$-estimates is as follows:
\begin{theorem}\label{t6.1}
Let $\Omega$ be a complete K\"{a}hler manifold with a bounded global K\"{a}hler potential $\varphi<0$, let $\alpha$ be a Hermitian metric on $\Omega$.
Fix a real $(1,1)$-form $\chi$ on $(\Omega,\alpha)$ such that $\lambda[\alpha^{i\bar{p}}\chi_{j\bar{p}}]\in \Gamma$, given $h\in C^1(\Omega \times\mathbb{R})$,
let $u\in C^2(\Omega)$ satisfy that
\begin{equation}\label{e6.1}
F(\alpha^{i\bar{p}}(\chi_{j\bar{p}}+u_{j\bar{p}}))=h(x,u)
\end{equation}
Suppose that the above equation satisfies conditions \emph{\textbf{(a1)}}-\emph{\textbf{(a5)}}, and assume that
\begin{equation}\label{en6.2}
\sup_{\Omega}\frac{|h(x,0)-F[\alpha^{i\bar{p}}\chi_{j\bar{p}}]|}{\underset{\min}{\lambda}(\alpha^{a\bar{l}}\varphi_{b\bar{l}})}<\infty,
\end{equation}
then there is a constant $\zeta$ depending only on $\Omega$, $\alpha$, $\chi$, $F$, $h$, $\varphi$ and $n$ such that
\begin{equation}\label{en6.3}
|u|\leq \zeta(1+\limsup_{z\rightarrow \infty}|u|).
\end{equation}
\end{theorem}

The following definition was introduced by Guan \cite{Guan14} and Sz\'{e}kelyhidi \cite{S18}, we restate it here for application in our case.

\begin{definition}\label{d4.1}
We say that $u$ is a $\mathcal{C}$-subsolution for the equation $F(A)=h$ (see \eqref{e1.1}) if the following holds. We require that for every point $x\in M$, if $\lambda=(\lambda_1,\ldots,\lambda_n)$ denote the eigenvalues of the endomorphism $\alpha^{i\bar{p}}g_{j\bar{p}}$ at $x$, then for all $i=1,\ldots,n$ we have
\begin{equation}\label{ne4.5}
\lim_{t\rightarrow\infty}f(\lambda+t\textbf{e}_i)>h(x)
\end{equation}
when $h(x)$ is independent of $u$. If $h(x,u)$ depends on both $x$ and $u$, we instead require
\begin{equation}
\lim_{t\rightarrow\infty}f(\lambda+t\textbf{{e}}_i)=\infty.
\end{equation}
Here $\textbf{e}_i$ denotes the $i$th standard basis vector. Note that part of the requirement is that $\lambda+t\textbf{e}_i\in \Gamma$ for sufficiently large $t$, for the limit to be defined.
\end{definition}

Cheng-Yau \cite{CY80} derived the a priori estimates for complex Monge-Amp\`{e}re equations on noncompact K\"{a}hler manifolds using generalized maximum principle. To solve the Gauduchon conjecture, Sz\'{e}kelyhidi-Tosatti-Weinkovewe \cite{STW17} established the second order estimates of a general class of fully non-linear equations on compact Hermitian manifolds. By combining the generalized maximum principle with certain results from \cite{STW17}, we obtain the a priori estimates for \eqref{e1.1} up to $C^{k+\beta}$:

\begin{theorem}\label{t4.4}
Let $(M,\alpha)$ be a complete Hermitian manifold with bounded geometry (see Definition \ref{db2.1}) of order $k-1$, $k\geq 4$. Fix a real $(1,1)$-form $\chi$ with bounded geometry of order $k-1$ on $(M,\alpha)$, given $h\in \tilde{C}^{k-2+\beta}(M)$,
let $u\in \tilde{C}^{k+\beta}(M)$ satisfy that
\begin{equation}\label{e1.9}
F(\alpha^{i\bar{p}}(\chi_{j\bar{p}}+u_{j\bar{p}}))=h,
\end{equation}
where $\beta\in (0,1)$. Suppose the $\underline{u}\in \tilde{C}^{\infty}(M)$ is a $\mathcal{C}$-subsolution for the equation \eqref{e1.9}. Suppose that the equation satisfies conditions \emph{\textbf{(a1)}}-\emph{\textbf{(a4)}}. Then, we have an estimate $|u|_{k+\beta}\leq C_k$, with constant $C_k$ depending on $k$, on the background date $M,\alpha,\chi,F,h$, on $\sup_M |u|$, and the subsolution $\underline{u}$.
\end{theorem}

Finally, we obtain the following existence result for the solution of $F(A)=h$:

\begin{theorem}\label{t7.4}
Let $\Omega$ be a complete K\"{a}hler manifold with a bounded global K\"{a}hler potential $\varphi<0$ and $\lim_{x\rightarrow\infty}\varphi(x)=0$. Let $\alpha$ be a Hermitian metric on $\Omega$ with bounded geometry of order $k-1$, where $k\geq 4$.
Fix a real $(1,1)$-form $\chi$ with bounded geometry of order $k-1$, and $\chi$ is uniformly equivalent to $\alpha$ on $(\Omega,\alpha)$ (see section 2), given $h\in \tilde{C}^{k-2+\beta}(\Omega \times\mathbb{R})$, where $\beta\in (0,1)$. Consider the following equation on $\Omega$:
\begin{equation}\label{e1.10}
F(\alpha^{i\bar{p}}(\chi_{j\bar{p}}+u_{j\bar{p}}))=h(x,u),
\end{equation}
Suppose that the above equation satisfies conditions \emph{\textbf{(a1)}}-\emph{\textbf{(a5)}} and $F(\alpha^{i\bar{p}}\chi_{j\bar{p}})=0$, assume that $|h(x,0)|=O(\underset{\min}{\lambda}(\alpha^{a\bar{l}}\varphi_{b\bar{l}}))$. Suppose $0$ is a $\mathcal{C}$-subsolution of \eqref{e1.10}. Then the equation \eqref{e1.10} exists a solution $u\in \tilde{C}^{k+\beta}(\Omega)$ such that $\chi+\sqrt{-1}\partial\bar{\partial}u$ is uniformly equivalent to $\alpha$.
\end{theorem}

If, in addition, assuming that $h_u(x,u)$ is strictly positive, we can generalize Theorem \ref{t7.4} to Hermitian manifolds:
\begin{theorem}\label{t7.5}
Let $(\Omega,\alpha)$ be a complete Hermitian manifold with bounded geometry of order $k-1$, where $k\geq 4$.
Fix a real $(1,1)$-form $\chi$ with bounded geometry of order $k-1$, and $\chi$ is uniformly equivalent to $\alpha$ on $(\Omega,\alpha)$, given $h\in \tilde{C}^{k-2+\beta}(\Omega \times\mathbb{R})$ such that $h_u\geq a>0$ for some constant $a$, where $\beta\in (0,1)$. Consider the following equation on $\Omega$:
\begin{equation}\label{e7.33}
F(\alpha^{i\bar{p}}(\chi_{j\bar{p}}+u_{j\bar{p}}))=h(x,u),
\end{equation}
Suppose that the above equation satisfies conditions \emph{\textbf{(a1)}}-\emph{\textbf{(a5)}} and $F(\alpha^{i\bar{p}}\chi_{j\bar{p}})=0$, suppose $0$ is a $\mathcal{C}$-subsolution of \eqref{e7.33}. Then the equation \eqref{e7.33} exists a solution $u\in \tilde{C}^{k+\beta}(\Omega)$ such that $\chi+\sqrt{-1}\partial\bar{\partial}u$ is uniformly equivalent to $\alpha$.
\end{theorem}

Apply Theorem \ref{t7.4} to Monge-Amp\`{e}re equations on strictly pseudoconvex domains, we obtain:

\begin{theorem}\label{t1.7}
Let $\Omega$ be a strictly pseudoconvex domain in a complete K\"{a}hler manifold $(M,h)$ as in Proposition \ref{pb2.3} with $l=\infty$. Define $g=-\log(-\varphi)$, where $\varphi$ is the defining function of $\Omega$, suppose that $|\partial\bar{\partial}\varphi|_{h}>c>0$. Fix a volume form $V (x)$ on $\Omega$ such that
\begin{equation}\label{e1.12}
|(\sqrt{-1})^n(V(x)-\det(g_{i\bar{j}}))dz^1 \wedge d\bar{z}^1\cdots \wedge dz^n\wedge d\bar{z}^n|_{h}=O(\varphi^{2(1-n)})
\end{equation}
Then there exists a complete K\"{a}hler metric $\alpha$ on $\Omega$ such that $\det(\alpha_{i\bar{j}})=V(x)$.
\end{theorem}

The above corollary can be used to find K\"{a}hler metrics with some prescribed volume forms on strictly pseudoconvex domains.

Denote the first Chern class of a K\"{a}hler manifold $M$ by $c_1(M)$. The existence of the K\"{a}hler-Einstein metrics on compact K\"{a}hler manifolds has been established in all cases: When $c_1(M)<0$, see Aubin \cite{Aubin76} and Yau \cite{Yau78}; When $c_1(M)=0$, see Yau \cite{Yau78}; When $c_1(M)>0$, see Chen-Donaldson-Sun \cite{CDS15a,CDS15b,CDS15c} and Tian \cite{Tian97,Tian15}. We obtain the following theorem on noncompact K\"{a}hler manifold:

\begin{theorem}\label{t1.8}
Let $(M,\alpha)$ be a complete K\"{a}hler manifold with negative Ricci curvature bounded from below by a negative
constant ($0>{\rm Ric}(\alpha)\geq c\alpha$ for some constant $c<0$). Then, for any compact subset $K\subset M$, there exists a K\"{a}hler-Einstein metric on $K$. Moreover, this metric can be extended to some open set $N\subset M$ and forms a complete Hermitian metric on $N$.
\end{theorem}

A related problem, the existence of Hesse-Einstein metrics ($\beta(g)=\lambda g$, see Section 8) on compact affine K\"{a}hler manifolds (Hessian manifolds) with negative first Chern class, was solved by Cheng-Yau \cite{CY82}. Puechmorel-T{\^o} \cite{PT23} used geometric flow to provide another method for constructing the above metrics. For more studies on affine differential geometry, see (\cite{CY77, HW17, LYZ05}).
In the following, we use the a priori estimates for \eqref{e1.1} to solve the real Monge-Amp\`{e}re equations and construct the Hesse-Einstein metrics on certain noncompact affine manifolds.
\begin{theorem}\label{t1.9}
Let $(\Omega,g)$ be a complete Hessian manifold with bounded geometry of infinity order. Suppose that the second Koszul form of $(\Omega,g)$ is bounded from below by a positive constant ($\kappa(g)>cg$ for some constant $c>0$), then $M$ admits a complete Hesse-Einstein metric.
\end{theorem}

The rest of this paper is organized as follows. In Section 2, we provide some preliminaries which may be used later.
In Section 3, we construct some complete metrics with bounded geometry. In Section 4, we establish the $C^0$ estimates. In Section 5, we prove theorem \ref{t4.4}. In Section 6, we solve the equation \eqref{e1.1}. In Section 7, we discuss the complex Monge-Amp\`{e}re equations and some geometric applications. In Section 8, we consider some applications of \eqref{e1.1} in affine geometry. We write an outline at the beginning of each section to introduce some key ideas.

\textbf{Acknowledgement.} I would like to thank Professor Wang Zhizhang for helpful
comments and discussions about the K\"{a}hler manifolds with bounded global K\"{a}hler potential. And I wish to thank Professor Jiao Heming suggesting him to read \cite{CY80}, \cite{STW17} and to study the $(n-1)$ Monge-Amp\`{e}re equations on noncompact Hermitian manifolds.

\section{Preliminaries}

In this section, we provide some basic results and preliminaries which may be used in the following sections. Some notations and conventions used in this paper will also be introduced at various points throughout this section.

$\bullet$
\begin{minipage}[t]{0.95\linewidth}
Let $(M,g)$ be a Hermitian manifold. Then \emph{Chern connection} of $g$ is defined as follows:  In local holomorphic
coordinates $z^i$, for a $T^{1,0}$ vector field $X=X^i \partial_i$ and a $T^{0,1}$ vector field $Y=Y^{\bar{i}} \partial_{\bar{i}}$, where $\partial_i:=\frac{\partial}{\partial z^i}$, $\partial_{\bar{i}}:=\frac{\partial}{\partial \bar{z}^i}$,
\[\nabla_k X^i=\partial_k X^i+\Gamma_{jk}^i X^j,\;\;\;\nabla_{\bar{k}}X^i=\partial_{\bar{k}}X^i\]
\[\nabla_k Y^{\bar{i}}=\partial_k Y^{\bar{i}},\;\;\;\nabla_{\bar{k}}Y^{\bar{i}}=\partial_{\bar{k}}Y^{\bar{i}}+\overline{\Gamma_{jk}^i}Y^{\bar{j}}\]
For a $(1,0)$ form $a=a_i {\rm d}z^i$ and a $(0,1)$ form $b_{\bar{i}}{\rm d}\bar{z^i}$,
\[\nabla_k a_i=\partial_k a_i-\Gamma_{ik}^j a_j,\;\;\;\nabla_{\bar{k}}a_i=\partial_{\bar{k}}a_i\]
\[\nabla_k b_{\bar{i}}=\partial_k b_{\bar{i}},\;\;\;\nabla_{\bar{k}}b_{\bar{i}}=\partial_{\bar{k}}b_{\bar{i}}-\overline{\Gamma^j_{ik}}b_{\bar{j}}\]
Here $\nabla_i:=\nabla_{\partial_i}$, and $\Gamma_{jk}^i$ are Christoffel symbols, with
\[\Gamma_{jk}^i=g^{i\bar{l}}\partial_j g_{k\bar{l}},\]
where $g^{i\bar{l}}g_{j\bar{l}}=\delta_{ij}$,
\[\delta_{ij}=\left\{ \begin{aligned}
&0, &\;\;\;i\neq j\\
&1,&\;\;\;i=j\\
\end{aligned} \right.\]
\end{minipage}

$\bullet$ The tensor operations in this paper follow the \emph{Einstein summation convention}.

$\bullet$
\begin{minipage}[t]{0.95\linewidth}
We extend covariant derivatives to act naturally on any type of tensor. For example, if $W$ is a tensor with components $W_{k}^{i\bar{j}}$, then define
\[\nabla_m W_{k}^{i\bar{j}}=\partial_m W_{k}^{i\bar{j}}+\Gamma_{lm}^i W_{k}^{l\bar{j}}-\Gamma_{km}^l W_{l}^{i\bar{j}}\]
\[\nabla_{\bar{m}}W_{k}^{i\bar{j}}=\partial_{\bar{m}}W_{k}^{i\bar{j}}+\overline{\Gamma_{lm}^j}W_{k}^{i\bar{l}}\]
\end{minipage}

$\bullet$
\begin{minipage}[t]{0.95\linewidth}
The Hermitian metric $g$ defines a pointwise norm $|\cdot|_g$ on any tensor. For example, with $X,Y,a,b$ as above, we define
\[|X|_g^2=g_{i\bar{j}}X^i\overline{X^j},\;\;|Y|_g^2=g_{i\bar{j}}Y^{\bar{j}}\overline{Y^{\bar{i}}},\;\;|a|_g^2=g^{i\bar{j}}a_i\overline{a_j},\;\;|b|_g^2=g^{i\bar{j}}b_{\bar{j}}\overline{b_{\bar{i}}}.\]
This is extended to any type of tensor. For example, with $W$ as above, we define
\[|W|^2_g=g^{k\bar{l}}g_{i\bar{j}}g_{p\bar{q}}W_k^{i\bar{q}}\overline{W_l^{j\bar{p}}}.\]
\end{minipage}

$\bullet$
\begin{minipage}[t]{0.95\linewidth}
Noted that the Chern connection is a connection such that $\nabla g=\nabla J=0$. The associated \emph{Chern curvature tensor} for the metric $g$ is then defined as
\[R_{i\bar{j}k}^{\;\;\;\;\;l}=-\partial_{\bar{j}}\Gamma_{ik}^l.\]
The Chern-Ricci curvature is defined by
\[R_{i\bar{j}}=g^{k\bar{l}}R_{i\bar{j}k\bar{l}}=-\partial_i\partial_{\bar{j}}\log\det g.\]
Note that if $g$ is not K\"{a}hler, then $R_{i\bar{j}}$ may not equal to $g^{k\bar{l}}R_{k\bar{l}i\bar{j}}$.
\end{minipage}

$\bullet$
\begin{minipage}[t]{0.95\linewidth}
For two real $(1,1)$-forms $\chi$ and $\chi'$ on $(M,g)$, we say that $\chi$ is uniformly equivalent to $\chi'$ if and only if
\[c^{-1}\chi'<\chi<c\chi'\]
for some constant $c>0$.
\end{minipage}

$\bullet$
\begin{minipage}[t]{0.95\linewidth}
Suppose that $F$ is a function defined by
\[F(B)=f(\lambda(B))\]
for a Hermitian matrix $B=\{B_{i\bar{j}}\}$ with $\lambda(B)\in \Gamma$, throughout the paper we shall use the notation
\[F^{ij}(B)=\frac{\partial F}{\partial B_{i\bar{j}}},\;\;\;\;\;F^{ij,kl}(B)=\frac{\partial^2 F}{\partial B_{i\bar{j}}\partial B_{k\bar{l}}},\]
where $B_{i\bar{j}}$ is a complex variable. If $f$ satisfies \textbf{(a1)} and \textbf{(a2)}, then $F(B)$ is concave (see for instance Spruck \cite{Spruck05} in the case of matrices with real entries). This means that:
\[F^{ij}(A)(B-A)_{i\bar{j}}\geq F(B)-F(A)\geq F^{ij}(B)(B-A)_{i\bar{j}},\]
for any Hermitian matrices $A$ and $B$ such that $\lambda(A),\lambda(B)\in \Gamma$.
\end{minipage}

$\bullet$
\begin{minipage}[t]{0.95\linewidth}
We denote the minimum and maximum eigenvalues of $B$ by $\underset{\min}{\lambda}(B)$ and $\underset{\max}{\lambda}(B)$, respectively.
\end{minipage}

$\bullet$
\begin{minipage}[t]{0.95\linewidth}
For two real functions $f$ and $h$, we denote $|f|=O(|h|)$ if there exists a constant $C>0$ such that $|f|\leq C|h|$.
\end{minipage}

$\bullet$
\begin{minipage}[t]{0.95\linewidth}
For a complete noncompact Hermitian manifold $(M,g)$, we say $z\rightarrow\infty$ if and only if $d_g(z,z_0)\rightarrow \infty$ for some fix point $z_0\in M$, where $d_g$ denotes the distance function induced by the metric $g$. Let $z_1\in M$ be another point, since $d_g(z,z_1)\geq d_g(z,z_0)-d_g(z_1,z_0)$, this definition is independent of the choice of $z_0$.
\end{minipage}

\section{Canonical metrics}

In this section, we construct some complete metrics with bounded geometry on a class of K\"{a}hler manifolds which admit a bounded global K\"{a}hler potential. With regard to the problem of a priori estimates for elliptic equations on non-compact manifolds, geometric boundedness is important, as this property can help us find a local coordinate system at each point on the manifold such that the coefficients in the equation satisfy certain uniform boundedness conditions under this coordinate system. And on such manifolds, we have the generalized maximum principle (see Corollary \ref{pbb4.3}).

Next, we will introduce some definitions and propositions about complete noncompact complex manifolds. To facilitate subsequent discussions, we introduce the following definition.
\begin{definition}\label{db2.1}
Suppose that $M$ is a complete Hermitian manifold. We say that $M$ has bounded geometry of order $l+\alpha$ if and only if $M$ admits a covering of holomorphic coordinate charts $\{(V,(v^1,\cdots,v^n))\}$ and positive numbers $R,c$ such that:

\hspace*{0.2cm}(i)
\begin{minipage}[t]{0.91\linewidth}
for any $x_0\in M$ there is a coordinate chart $(V,(v^1,\cdots,v^n))$ with $x_0\in V$ and that, with respect to the Euclidean distance $d$ defined by $v^i$-coordinates, $d(x_0,\partial V)>R$;
\end{minipage}

\hspace*{0.1cm}(ii)
\begin{minipage}[t]{0.91\linewidth}
if $(g_{i\bar{j}})$ denote the metric tensor with respect to $(V,(v^1,\cdots,v^n))$, then the components $g_{i\bar{j}}$ of $g$ are uniformly bounded in the standard $C^{l+\alpha}$ norm in $(V,(v^1,\cdots,v^n))$ independent of $V$ and $(\delta_{ij})/c<(g_{i\bar{j}})<c(\delta_{ij})$.
\end{minipage}

Fix a real $(1,1)$-form $\chi$ on $M$, we say that $\chi$ has bounded geometry of order $l$ with respect to $(M,g_{i\bar{j}})$ if and only if: In the coordinate chart $(V,(v^1,\cdots,v^n))$, the components $\chi_{i\bar{j}}$ of $\chi$ are uniformly bounded in the standard $C^{l+\alpha}$ norm in $(V,(v^1,\cdots,v^n))$ independent of $V$.
\end{definition}

Suppose $(M,g_{i\bar{j}})$ is a complete Hermitian manifold having bounded geometry of order $l+\beta$. Let $\{(V,(v^1,\cdots,v^n))\}$ be a family of holomorphic coordinate charts covering $M$ and satisfying the conditions of Definition \ref{db2.1}. Then, for any $u\in C^\infty(M)$, non-negative integer $k,k\leq l$, and $\alpha \in(0,1)$, we define a norm $|u|_{k+\alpha}$ to be
\begin{equation}\label{eb2.18}
\begin{aligned}
|u|_{k+\alpha}=&\sup_V\biggl(\sup_{z\in V}\biggl(\sum_{|\alpha|+|\beta|\leq k}\biggl|\frac{\partial^{|\alpha|+|\beta|}}{\partial v^\alpha \partial \bar{v}^\beta}u(z)\biggl|\biggl)\\
&+\sup_{z,z'\in V}\biggl(\sum_{|\alpha|+|\beta|= k}|z-z'|^{-\alpha}\biggl|\frac{\partial^{|\alpha|+|\beta|}}{\partial v^\alpha \partial \bar{v}^\beta}u(z)-\frac{\partial^{|\alpha|+|\beta|}}{\partial v^\alpha \partial \bar{v}^\beta}u(z')\biggl|\biggl)\biggl).
\end{aligned}
\end{equation}
The completion of $\{u\in C^\infty(M):|u|_{k+\alpha}<\infty\}$ with respect to $|\cdot|_{k+\alpha}$ is then a Banach space and will be denoted by $\tilde{C}^{k+\alpha}(M)$. Define,
\begin{equation}
\tilde{C}^\infty(M)=\bigcap_{k=0}^\infty \tilde{C}^{k+\alpha}(M).
\end{equation}
\textbf{Remark}: If $M$ has bounded geometry of infinity order and admits a covering of holomorphic coordinate charts $\{(V,(v^1,\cdots,v^n))\}$. Then, for any $h\in C^\infty(M\times\mathbb{R})$, we can take $\{(V\times\mathbb{R},(v^1,\cdots,v^n,x))\}$ as the covering of coordinate charts for $M\times\mathbb{R}$, and define $|h|_{k+\alpha}$, $\tilde{C}^{k+\alpha}(M\times\mathbb{R})$ and $\tilde{C}^\infty(M\times\mathbb{R})$ similarly.

On the manifolds with bounded geometry, we have the generalized maximum principle (see \cite{CY80}), which is our main tool to derive the a priori estimates. For completeness, we state this in the following proposition.
\begin{proposition}\label{pb4.2}
Suppose $(M,g_{i\bar{j}})$ is a complete K\"{a}hler manifold. Suppose that, for any $x\in M$, there is an open set $D^x$ containing $x$ and non-negative function $\varphi^x:\overline{D^x}\rightarrow \mathbb{R}$ such that {\rm (i)} $\overline{D^x}$ is compact, {\rm (ii)} $\varphi^x_{(x)}=1$ and $\varphi^x=0$ on $\partial D^x$, {\rm (iii)} $\varphi^x\leq c$, $|\nabla \varphi^x|\leq c$ and $(\varphi_{i\bar{j}}^x)\geq -c(g_{i\bar{j}})$, where $c$ is a positive constant independent of $x$. Suppose $f$ is a function on $M$ which is bounded from above. Then there exists a sequence $\{x_i\}$ in $M$ such that $\lim f(x_i)=\sup f$, $\lim|df(x_i)|=0$ and $\overline{\lim}(f_{p\bar{q}}(x_i))\leq 0$, where the Hessian is taken with respect to $(g_{i\bar{j}})$.
\end{proposition}
\begin{corollary}\label{pbb4.3}
Suppose $(M,g_{i\bar{j}})$ is a complete K\"{a}hler manifold with bounded geometry of order $l$, $l\geq 0$. Then the assertion of Proposition \ref{pb4.2} is valid on $(M,g_{i\bar{j}})$.
\end{corollary}

Next, we will consider some manifolds that have bounded geometry of order $l$. Let $\Omega$ be a $C^k$, $k\geq 5$, strictly pseudoconvex domain in $\mathbb{C}^n$. Let $\varphi$ be a $C^k$ defining function for $\Omega$. i.e., $\varphi\in C^k(\bar{\Omega})$, $\varphi=0$ on $\partial \Omega$, $d\varphi\neq 0$ on $\partial \Omega$, $(\varphi_{i\bar{j}})>0$ in $\bar{\Omega}$ and $\Omega=\{\varphi<0\}$. Define $g=-\log(-\varphi)$. Then $g$ is a strictly plurisubharmonic function defined in $\Omega$ and $g(x)\rightarrow \infty$ as $x\rightarrow \partial \Omega$. As easy calculation shows
\begin{equation}
g_{i\bar{j}}=\frac{\varphi_{i\bar{j}}}{-\varphi}+\frac{\varphi_i\varphi_{\bar{j}}}{\varphi^2}
\end{equation}
and
\begin{equation}
g^{i\bar{j}}=(-\varphi)\biggl(\varphi^{i\bar{j}}+\frac{\varphi^i\varphi^{\bar{j}}}{\varphi-|d\varphi|^2}\biggl),
\end{equation}
where $(\varphi^{i\bar{j}})=(\varphi_{i\bar{j}})^{-1}$, $\varphi^i=\sum \varphi^{i\bar{k}}\varphi_{\bar{k}}$, and $|d\varphi|^2=\sum\varphi^{i\bar{j}}\varphi_i\varphi_{\bar{j}}$.

By Cheng-Yau \cite{CY80}, we have Proposition \ref{pb2.2} and \ref{pb2.3}.
\begin{proposition}\label{pb2.2}
Suppose $\Omega$ is a $C^{l+2}$ strongly pseudoconvex domain in $\mathbb{C}^n$. Let $\varphi$ be a defining function of $\Omega$ and $u\in C^{l+2}(\bar{\Omega})$ such that $(u_{i\bar{j}})>0$. Then $(\Omega,-(\log(-\varphi))_{i\bar{j}}+u_{i\bar{j}})$ has bounded geometry of order $l$.
\end{proposition}
\begin{proposition}\label{pb2.3}
Suppose $(M,h_{i\bar{j}})$ is a $C^{l+2}$ K\"{a}hler manifold and $\Omega\subset\subset M$ is a $C^{l+2}$ strictly pseudoconvex domain of $M$. Let $\varphi$ be a defining function of $\Omega$, i.e., $\varphi\in C^{l+2}(\bar{\Omega})$, $\varphi=0$ on $\partial\Omega$, $d\varphi\neq 0$ on $\partial \Omega$, $(\varphi_{i\bar{j}})\geq0$, $(\varphi_{i\bar{j}})>0$ in a neighborhood of $\partial\Omega$ and $\Omega=\{\varphi<0\}$. Then, $(\Omega,-(\log(-\varphi))_{i\bar{j}}+h_{i\bar{j}})$ has bounded geometry of order $l$.
\end{proposition}
Proposition \ref{pb2.2} and \ref{pb2.3} are the examples of some relatively compact pseudoconvex domains of K\"{a}hler manifolds which have bounded geometry of order $l$. In this paper, we also consider a class of domains which is not necessary to be relatively compact on K\"{a}hler manifolds. We have the following:
\begin{proposition}\label{pb2.4}
Suppose $(M,h_{i\bar{j}})$ is a K\"{a}hler manifold with bounded geometry of order $l+2$ and $\Omega\subset M$ is a $C^{l+2}$ open domain of $M$ satisfy the following:

\hspace*{0.2cm}{\rm (i)}
\begin{minipage}[t]{0.91\linewidth}
there exists a defining function $\varphi$ on $\Omega$ such that, $\varphi\in \tilde{C}^{l+2}(\bar{\Omega})$, $\varphi=0$ on $\partial\Omega:=\bar{\Omega} \setminus \Omega$, $|d\varphi|_h\geq c$ for some $c>0$ on $\partial \Omega$, $(\tilde{\nabla}_p\tilde{\nabla}_q\varphi)\geq 0$$(1\leq p,q\leq 2n)$ in $\Omega$ and $\Omega=\{\varphi<0\}$ (where $\tilde{\nabla}$ denotes the Levi-Civita connection of the Riemannian metric with respect to $h_{i\bar{j}}$);
\end{minipage}

\hspace*{0.1cm}{\rm (ii)}
\begin{minipage}[t]{0.91\linewidth}
for any $x\in\partial\Omega$, let $(V,(v^1,\cdots,v^n))$ is a coordinate chart around $x$ with respect to the metric $h_{i\bar{j}}$ as in Definition \ref{db2.1}. In this coordinate chart, there is a positive number $\delta$ such that there exists a ball of radius $\delta$ lies inside $\Omega$ which touches to $\partial\Omega$ at $x$, and $(\varphi_{i\bar{j}})>0$ on this ball. Where $\delta$ is independent of the choice of $x$.
\end{minipage}

Then, $(\Omega,-(\log(-\varphi))_{i\bar{j}}+h_{i\bar{j}})$ has bounded geometry of order $l$.(If $\bar{\Omega}$ is compact, we may weaken condition $(\tilde{\nabla}_p\tilde{\nabla}_q\varphi)\geq 0$ to $(\varphi_{i\bar{j}})\geq0$).
\end{proposition}
\begin{proof}
Since $(M,h_{i\bar{j}})$ is a K\"{a}hler manifold, at any point, we can choose a holomorphic coordinate such that all the first-order derivatives of $h_{i\bar{j}}$ vanish, so $(\tilde{\nabla}_p\tilde{\nabla}_q\varphi)=(\varphi_{qp})\geq 0$. Then by a unitary transformation, $(\varphi_{i\bar{j}})$ can be diagonalized. Note that
\[\varphi_{i\bar{i}}=\frac{1}{4}(\frac{\partial}{\partial x^2_i}+\frac{\partial}{\partial y^2_i})\varphi.\]
Hence, the convexity of $\varphi$ implies that
$(\varphi_{i\bar{j}})\geq0$ on $\Omega$.

First, we need to prove that $(\Omega,-(\log(-\varphi))_{i\bar{j}}+h_{i\bar{j}})$ is a complete K\"{a}hler manifold. For convenience, we denote the K\"{a}hler metric $-(\log(-\varphi))_{i\bar{j}}+h_{i\bar{j}}$ by $\tilde{g}$. Denote the set of points at infinity of $M$ and $\Omega$ by $\infty_M$ and $\infty_\Omega$ respectively, and $\infty_\Omega$ can be divided into two parts: $\partial\Omega=\{x\in M|\varphi(x)=0\}$ and $\Omega\cap \infty_M$. According to the Hopf-Rinow theorem, the completeness of the metric on a Riemannian manifold is equivalent to the condition that every bounded closed subset of the manifold is compact. Therefore, it suffices to prove
\begin{equation}
d_{\tilde{g}}(x_0,\partial \Omega\cup (\Omega\cap \infty_M))=\infty\;\;\;\;{\rm for} \;\forall x_0\in \Omega,
\end{equation}
where $d_{\tilde{g}}$ is the distance function with respect to the metric $\tilde{g}$. Since $(M,h_{i\bar{j}})$ is complete, we have
\begin{equation}\label{eb2.22}
d_{\tilde{g}}(x_0,\Omega\cap \infty_M)\geq d_h(x_0,\Omega\cap \infty_M)=\infty.
\end{equation}

Then, we calculate $d_{\tilde{g}}(x_0,\partial \Omega)$. Since $(M,h_{i\bar{j}})$ is a K\"{a}hler manifold with bounded geometry of order $l+2$, let $(V,(v^1,\cdots,v^n))$ be a holomorphic coordinate chart around $x_0$ as in Definition \ref{db2.1}, we calculate in this coordinate chart. Let $a$ be some point on $\partial\Omega$ and $\gamma(t):[0.1]\rightarrow \bar{\Omega}$ is a curve such that $\gamma(0)=a$, $\gamma(1)=x_0$ and $|\gamma'(t)|=1$.

Next, we calculate the length of the curve $\gamma(t)$, which we denote by $L(\gamma(t))$,
\begin{equation}
\begin{aligned}\label{ne2.23}
L(\gamma(t))&=\int_0^1\sqrt{\langle\gamma'(t),\gamma'(t)\rangle_g}{\rm d}t\\
&= \int_0^1\sqrt{\sum_{i,j}\gamma_i'(t)g_{i\bar{j}}\overline{\gamma_j'(t)}}{\rm d}t\\
&\geq \int_0^1\sqrt{\sum_{i,j}\gamma_i'(t)\frac{\varphi_i\varphi_{\bar{j}}}{\varphi^2}\overline{\gamma_j'(t)}}{\rm d}t.\\
\end{aligned}
\end{equation}

First, we bound the integrand in the above integral \eqref{ne2.23}. For any point~$p$ on the curve (we may assume that the entire curve lies in a sufficiently small neighborhood of~$a$), we can compute the gradient of~$\varphi$ in real coordinates and denote the unit vector corresponding to the gradient as~$\nabla\varphi(p)$. Clearly, there exists a unitary transformation~$F$ such that~${\rm d}F(\partial/\partial x^1)=\nabla\varphi(p)$, and using this unitary transformation, we can define a new holomorphic coordinate system~$(v^1\circ F^{-1},\cdots,v^n\circ F^{-1})$ at the point~$p$. In this coordinate system, when~$i\neq 1$, $\varphi_i=0$, from which it follows that
\begin{equation}
\begin{aligned}\label{ne2.24}
\sqrt{\sum_{i,j}\gamma_i'(t)\frac{\varphi_i\varphi_{\bar{j}}}{\varphi^2}\overline{\gamma_j'(t)}}&=\biggl|\frac{{\rm d}v^1\circ F^{-1}(\gamma(t))}{{\rm d}t}\frac{\varphi_1}{\varphi}\biggl|\\
&=\biggl|\biggl(\frac{\partial\varphi}{\partial v^1\circ F^{-1}}(\gamma(t))\biggl)^{-1}\frac{{\rm d}\varphi(\gamma(t))}{{\rm d}t}\frac{\varphi_1}{\varphi}\biggl|\\
&\geq \frac{c}{\varphi}\frac{{\rm d}\varphi(\gamma(t))}{{\rm d}t},
\end{aligned}
\end{equation}
where~$c$ is a positive constant independent of the choice of~$p$, by \eqref{ne2.23} and \eqref{ne2.24}, we have
\begin{equation}
\begin{aligned}
L(\gamma(t))&\geq\int_0^1 \frac{c}{\varphi}\frac{{\rm d}\varphi(\gamma(t))}{{\rm d}t}{\rm d}t\\
&=\int_{\varphi(a)}^{\varphi(x_0)}\frac{c}{\varphi}{\rm d}\varphi\\
&=\infty.
\end{aligned}
\end{equation}

Hence
\begin{equation}\label{eb2.29}
d_{\tilde{g}}(x_0,\partial \Omega)=\infty,
\end{equation}
by \eqref{eb2.22} and \eqref{eb2.29}, we can prove that $(\Omega,-(\log(-\varphi))_{i\bar{j}}+h_{i\bar{j}})$ is a complete K\"{a}hler manifold.

Next, we prove that $(\Omega,-(\log(-\varphi))_{i\bar{j}}+h_{i\bar{j}})$ has bounded geometry of order $l$. Define
\begin{equation}
\Omega_\epsilon:=\{x\in\Omega|d_h(x,\partial \Omega)<\epsilon\},
\end{equation}
if $\epsilon$ is sufficiently small, by Cheng-Yau \cite{CY80}, we have that for any $x\in \Omega_\epsilon$, we can find a coordinate chart around $x$ with respect to the metric $-(\log(-\varphi))_{i\bar{j}}+h_{i\bar{j}}$ as in Definition \ref{db2.1}. For $x\in \Omega \setminus \Omega_\epsilon$, we want to prove the coordinate chart $(V,(v^1,\cdots,v^n))$ around $x$ as above is also the coordinate chart with respect to the metric $-(\log(-\varphi))_{i\bar{j}}+h_{i\bar{j}}$ as in Definition \ref{db2.1}.

Obviously
\begin{equation}
-(\log(-\varphi))_{i\bar{j}}+h_{i\bar{j}}\geq h_{i\bar{j}}\geq c_2 \delta_{ij},
\end{equation}
for some $c_2>0$ which is dependent of the choice of $x$. Since $\varphi$ and $h_{i\bar{j}}$ have all their derivatives up to order $l+2$ and $l$ bounded by some constant independent of $x$, and
\begin{equation}
-(\log(-\varphi))_{i\bar{j}}=\frac{\varphi_{i\bar{j}}}{-\varphi}+\frac{\varphi_i\varphi_{\bar{j}}}{\varphi^2},
\end{equation}
then we only need to prove that
\begin{equation}\label{eb2.33}
\varphi< -c_3 \;\;\;\;\;{\rm on}\; \Omega \setminus \Omega_\epsilon,
\end{equation}
for some $c_3>0$. If $\bar{\Omega}$ is compact, the fact \eqref{eb2.33} is obvious, in the following, we consider the case that $\bar{\Omega}$ is noncompact. Since $|d\varphi|\geq c$, if $\epsilon$ is sufficiently small, we have that
\begin{equation}\label{eb2.34}
\varphi(x)<-c_4 \;\;\;\;\;{\rm for}\; x\in \partial \Omega_\epsilon\setminus\partial\Omega,
\end{equation}
for some $c_4>0$ independent of $x$. If there is no $c_3$ satisfying \eqref{eb2.33}, then we obtain
\begin{equation}\label{eb2.35}
\varphi(p)>-\frac{1}{2}c_4 \;\;\;\;\;{\rm for \;some}\; p\in \Omega \setminus \Omega_\epsilon.
\end{equation}
Let $p_1$ be a point on $\partial\Omega$ such that
\begin{equation}
d_h(p,p_1)=d_h(p,\partial\Omega).
\end{equation}

Let $(V,(x^1,\cdots,x^n))$ be the coordinate chart around $p$ with respect to the metric $h_{i\bar{j}}$ as in Definition \ref{db2.1}, then there exist a $v\in T_{p_1}M$ and a geodesic $\exp_{p_1} tv:[0,\infty)\rightarrow M$ defined by exponential map such that $\exp_{p_1} 0=p_1$ and $\exp_{p_1} t_0v=p$ for some $t_0>0$ and $L(\exp_{p_1} tv|_0^{t_0})=d_h(p,\partial\Omega)$. For any $t_1\geq 0$, we can find a local coordinate chart $(V,(x_1,\cdots,x_{2n}))$ around $\exp_{p_1} t_1v$ such that
\begin{equation}
\exp_{p_1} tv=(a_1(t-t_1),\cdots,a_{2n}(t-t_1))\;\;\;\;{\rm in}\;(V,(x_1,\cdots,x_{2n})),
\end{equation}
for some constants $a_1,\cdots,a_{2n}$, then we obtain
\begin{equation}\label{eb2.38}
\frac{d^2(\varphi(\exp_{p_1} tv))}{dt^2}=\sum_{p,q=1}^{2n}\varphi_{pq}a_pa_q\geq 0.
\end{equation}
By \eqref{eb2.38}, we have that
\begin{equation}\label{eb2.39}
\frac{d^2(\varphi(\exp_{p_1} tv))}{dt^2}\geq 0,
\end{equation}
for $t\in[0,\infty)$. Let $t_2$ is a constant such that $\exp_{p_1} t_2v\in \partial \Omega_\epsilon\setminus\partial\Omega$ and $t_2\in (0,t_0)$, then by \eqref{eb2.34}, we obtain
\begin{equation}\label{eb2.40}
\varphi(\exp_{p_1} t_2v))<-c_4,
\end{equation}
and by \eqref{eb2.35}, we have
\begin{equation}\label{eb2.41}
\varphi(\exp_{p_1} t_0v))>-\frac{1}{2}c_4.
\end{equation}
By \eqref{eb2.39}, \eqref{eb2.40} and \eqref{eb2.41}, we have that there exists $t_3\in(t_0,2t_0-t_2)$ such that
\begin{equation}
\varphi(\exp_{p_1} t_3v))=0,
\end{equation}
so $\exp_{p_1} t_3v\in \partial \Omega$, then we get
\begin{equation}
d_h(p,\partial \Omega)\leq L(\exp_{p_1} tv|_{t_0}^{t_3})<L(\exp_{p_1} tv|_0^{t_0})=d_h(p,\partial\Omega).
\end{equation}
This is a contradiction! Then we can prove \eqref{eb2.33}.
\end{proof}

We note that there exists a class of pseudoconvex domains in K\"{a}hler manifolds $(M,h_{i\bar{j}})$ as in Proposition \ref{pb2.3} or \ref{pb2.4} such that its defining function $\varphi$, $|\partial \varphi|_h$ and $|\partial\bar{\partial}\varphi|_h$ are bounded (these conditions can be used in $C^{0}$-estimates), as it is illustrated
by the following examples.
\begin{example}
Let $\Omega$ be a relatively compact strictly pseudoconvex domain in $\mathbb{C}^n$. It is well known the defining function $\varphi$ of $\Omega$ always exist. It is easy to see that $\varphi$, $|\partial \varphi|_{g_e}$ and $|\partial\bar{\partial}\varphi|_{g_e}$ are bounded on $\Omega$, where $g_e$ is the standard on $\mathbb{C}^n$. This is also a example of Proposition \ref{pb2.3}.
\end{example}
There are also some examples of unbounded pseudoconvex domains:
\begin{example}
Define
\[\Omega:=\{(z_1,\ldots,z_n)\in\mathbb{C}^n|\varphi(z_1,\ldots,z_n)<0\},\]
where
\begin{equation}\label{e2.41}
\varphi(z_1,\ldots,z_n)=\sum_{i=1}^n\biggl(-\arctan\Big(({\rm Im} z_i)+\tan\frac{n\pi}{2n+2}\Big)+\frac{n\pi}{2n+2}\biggl).
\end{equation}
So on $\Omega$,
\begin{equation}\label{e2.42}
\begin{aligned}
-\arctan\Big(({\rm Im} z_j)+\tan\frac{n\pi}{2n+2}\Big)&<\sum_{i\neq j}\biggl(\arctan\Big(({\rm Im} z_i)+\tan\frac{n\pi}{2n+2}\Big)-\frac{n\pi}{2n+2}\biggl)-\frac{n\pi}{2n+2}\\
&<\sum_{i\neq j}\biggl(\frac{\pi}{2}-\frac{n\pi}{2n+2}\biggl)-\frac{n\pi}{2n+2}\\
&=-\frac{\pi}{2n+2}
\end{aligned}
\end{equation}
for $j=1,\ldots,n$.
By straightforward calculations, we obtain
\begin{equation}\label{e2.43}
\varphi_i=\frac{\sqrt{-1}}{2}\frac{1}{1+\Big(({\rm Im} z_i)+\tan\frac{n\pi}{2n+2}\Big)^2},
\end{equation}
\begin{equation}\label{e2.44}
\varphi_{i\bar{j}}=\delta_{ij}\frac{\Big(({\rm Im} z_i)+\tan\frac{n\pi}{2n+2}\Big)}{2\biggl(1+\Big(({\rm Im} z_i)+\tan\frac{n\pi}{2n+2}\Big)^2\biggl)^2},
\end{equation}
By \eqref{e2.41}, \eqref{e2.42}, \eqref{e2.43} and \eqref{e2.44}, we obtain that $\varphi$ is a defining function of pseudoconvex domain $\Omega$, $\varphi$, $|\partial \varphi|_{g_e}$ and $|\partial\bar{\partial}\varphi|_{g_e}$ are bounded on $\Omega$. This a also a example of Proposition \ref{pb2.4}.
\end{example}

We can also give an example of a pseudoconvex domains on a non-K\"{a}hler manifold.
\begin{example}
$M$ is a non-K\"{a}hler manifold, and $h_{i\bar{j}}$ is a Hermitian metric on $M$. Then there must exist a point $p\in M$ such that $h_{i\bar{j}}(p)$ is not a K\"{a}hler metric. Take a local holomorphic coordinate system chart $(V(p),(z^1,\cdots,z^n))$ around $p$ and assume that $p$ is coordinate origin. Define $\varphi(z)=z\bar{z}-\epsilon$, where $\epsilon$ is a positive constant. Then if $\epsilon$ small enough, $\varphi$ can be seen as a defining function of a domain $\Omega:=\{z\in(V(p),(z^1,\cdots,z^n))|\varphi(z)<0\}$ on $(M,h_{i\bar{j}})$, $\varphi$, $|\partial \varphi|_h$ and $|\partial\bar{\partial}\varphi|_h$ are bounded on $\Omega$.
\end{example}


\section{$C^{0}$-estimates}

In this section, we want to establish a priori $C^{0}$-estimates of \eqref{e1.1}. One difficulty is that the equation's left-hand side does not involve $u$ explicitly. But note that $\partial\bar{\partial}(\varphi u)=u\partial\bar{\partial}\varphi+2{\rm Re}(\sqrt{-1}\partial u\wedge \bar{\partial}\varphi)+\varphi\partial\bar{\partial}u$. Therefore, if we replace the term $u$ in equation \eqref{e1.1} with $\varphi u$, we can apply the maximum principle to $u$ to establish a priori $C^{0}$-estimates, this is the reason we assume the K\"{a}hler manifold $\Omega$ admits a bounded global K\"{a}hler potential $\varphi$.

Our main Theorem of this Section is Theorem \ref{t6.1}. To prove this theorem, we first need the following lemma:

\begin{lemma}\label{l6.2}
Let $\Omega$ be a complete K\"{a}hler manifold with a bounded global K\"{a}hler potential $\varphi<0$, let $\alpha$ be a Hermitian metric on $\Omega$.
Fix a real $(1,1)$-form $\chi$ on $(\Omega,\alpha)$ such that $\lambda[\alpha^{i\bar{p}}\chi_{j\bar{p}}]\in \Gamma$, given $h\in C^1(\Omega\times\mathbb{R})$,
let $u\in C^2(\Omega)$ satisfy that
\begin{equation}\label{e6.3}
F[\alpha^{i\bar{p}}(\chi+\sqrt{-1}\partial\bar{\partial}(\varphi u-u))_{j\bar{p}}]=h(x,\varphi u-u).
\end{equation}
Suppose that the above equation satisfies conditions \emph{\textbf{(a1)}}-\emph{\textbf{(a5)}}, then there is a constant $c$ depending only on $\alpha$, $\chi$, $F$, $h$ and $n$ such that
\begin{equation}\label{e6.4}
\sup_{\Omega}|u|\leq c\sup_{\Omega}\frac{|h(x,0)-F[\alpha^{i\bar{p}}\chi_{j\bar{p}}]|}{\underset{\min}{\lambda}(\alpha^{a\bar{l}}\varphi_{b\bar{l}})}+\limsup_{z\rightarrow \infty}|u|.
\end{equation}
\end{lemma}
\begin{proof}
First, we want to estimate $\sup_{\Omega}u$. We divided into two cases to consider:(i) $\sup_{\Omega}u=\limsup_{z\rightarrow  \infty}|u|$, in this case, \eqref{e6.4} obviously holds; (ii) $\sup_{\Omega}u>\limsup_{z\rightarrow \infty}|u|$, then there exists some $P\in \Omega$ such that $u(P)=\sup_{\Omega}u$. By maximum principle, we have at $P$ that
\begin{equation}\label{e6.5}
\partial u=0,\;\;\;
\sqrt{-1}\partial\bar{\partial} u\leq 0.
\end{equation}
By \eqref{e6.3}, \eqref{e6.5} and the fact that \eqref{e6.3} is an elliptic equation, we get
\begin{equation}\label{e6.6}
\begin{aligned}
h(x,\varphi u-u)&=F[\alpha^{i\bar{p}}(\chi+\sqrt{-1}\partial\bar{\partial}(\varphi u-u))_{j\bar{p}}]\\
&=F[\alpha^{i\bar{p}}(\chi+\sqrt{-1}u\partial\bar{\partial}\varphi+2{\rm Re}(\sqrt{-1}\partial u\wedge \bar{\partial}\varphi)
+\sqrt{-1}(\varphi-1)\partial\bar{\partial} u)_{j\bar{p}}]\\
&\geq F[\alpha^{i\bar{p}}(\chi+\sqrt{-1}u\partial\bar{\partial}\varphi)_{j\bar{p}}].
\end{aligned}
\end{equation}
By condition \textbf{(a5)}, we obtain
\begin{equation}\label{en6.8}
h(x,\varphi u-u)\leq h(x,0)
\end{equation}
Hence $F(A)=f(\lambda_1,\ldots,\lambda_n)$ is concave, we have
\begin{equation}\label{e6.7}
\begin{aligned}
&F[\alpha^{i\bar{p}}(\chi+\sqrt{-1}u\partial\bar{\partial}\varphi)_{j\bar{p}}]-F[\alpha^{i\bar{p}}\chi_{j\bar{p}}]\\
\geq &F^{ab}[\alpha^{i\bar{p}}(\chi+\sqrt{-1}u\partial\bar{\partial}\varphi)_{j\bar{p}}](\alpha^{a\bar{l}}u\varphi_{b\bar{l}}),
\end{aligned}
\end{equation}
where $F^{pq}[B]$ denotes the partial derivative of the function $F(B)$ with respect to the $(p,q)$-entry of the metric $B$. By \eqref{e6.6}, \eqref{en6.8} and \eqref{e6.7}, we obtain
\begin{equation}\label{e6.8}
F^{ab}[\alpha^{i\bar{p}}(\chi+\sqrt{-1}u\partial\bar{\partial}\varphi)_{j\bar{p}}](\alpha^{a\bar{l}}\varphi_{b\bar{l}})u\leq h(x,0)-F[\alpha^{i\bar{p}}\chi_{j\bar{p}}]
\end{equation}

To estimate $\sup_{\Omega}u$ from \eqref{e6.8}, we will give lower bound for $\mathcal{F}$, the method is similar to Lemma 9 in \cite{S18}, where $\mathcal{F}=\sum_k F^{kk}[\alpha^{i\bar{p}}(\chi+\sqrt{-1}u\partial\bar{\partial}\varphi)_{j\bar{p}}]$. We can choose a new coordinate system at $P$ such that $\alpha_{i\bar{j}}=\delta_{ij}$ and $(\chi+\sqrt{-1}u\partial\bar{\partial}\varphi)_{j\bar{p}}=a_j \delta_{jp}$, then $[\alpha^{i\bar{p}}(\chi+\sqrt{-1}u\partial\bar{\partial}\varphi)_{j\bar{p}}]$ is a diagonal matrix, this follows that
\begin{equation}\label{e6.9}
\lambda_i^{pq}=\delta_{pi}\delta_{qi},
\end{equation}
where $\lambda_i$ is the $i$-th eigenvalue of the metric $[\alpha^{i\bar{p}}(\chi+\sqrt{-1}u\partial\bar{\partial}\varphi)_{j\bar{p}}]$, and $\lambda_i^{pq}$ denotes the derivative with respect to the $(p,q)$-entry of the metric $[\alpha^{i\bar{p}}(\chi+\sqrt{-1}u\partial\bar{\partial}\varphi)_{j\bar{p}}]$. By \eqref{e6.9}, we have
\begin{equation}
\begin{aligned}
F^{ab}[\alpha^{i\bar{p}}(\chi+\sqrt{-1}u\partial\bar{\partial}\varphi)_{j\bar{p}}]&=f^{ab}(\lambda_1,\ldots,\lambda_n)\\
&=\sum_k f_k \lambda_k^{ab}\\
&=\sum_k f_k \delta_{ak}\delta_{bk},
\end{aligned}
\end{equation}
Hence we have that $F^{ab}$ vanishes whenever $a\neq b$, while on the other hand $F^{kk}=f_k$. Then
\begin{equation}
\mathcal{F}=\sum_k f_k.
\end{equation}

By \textbf{(a5)}, we can choose $\sigma\in (\sup_\Omega h(x,0),\sup_\Gamma f)$, let $x$ be the closest point to the origin such that $f(x)=\sigma$. By the concavity of $f$ and symmetry under permuting the variables, we must have $x=N{\rm \textbf{1}}$ for some $N>0$, where ${\rm \textbf{1}}=(1,\ldots,1)\in \mathbb{R}^n$. For any $\lambda\in \Gamma$, assumption \textbf{(a4)} implies that there is some $T>1$, such that $f(T\lambda)>\sigma$. By \textbf{(a2)} condition, we also have $f(Tx/(T-1))>\sigma$. The concavity of $f$ implies that then
\begin{equation}\label{e6.12}
\begin{aligned}
f(x+\lambda)&=f(\frac{T-1}{T}\frac{Tx}{T-1}+\frac{1}{T}T\lambda)\\
&>\frac{T-1}{T}f(\frac{Tx}{T-1})+\frac{1}{T}f(T\lambda)\\
&>\sigma.
\end{aligned}
\end{equation}
Then by \eqref{e6.12} we have $f(\lambda+N{\rm \textbf{1}})>\sigma$. By concavity we obtain
\begin{equation}\label{en6.15}
f(\lambda+N{\rm \textbf{1}})\leq f(\lambda)+N\sum_{i=1}^n f_i(\lambda),
\end{equation}
and by \eqref{e6.6}, we have
\begin{equation}\label{en6.16}
f(\lambda(\alpha^{i\bar{p}}(\chi+\sqrt{-1}u\partial\bar{\partial}\varphi)_{j\bar{p}}))\leq h(P,\varphi u(P)-u(P))\leq h(P,0)\leq\sup_\Omega h(x,0),
\end{equation}
then by \eqref{en6.15} and \eqref{en6.16}, we have
\begin{equation}\label{e6.14}
\mathcal{F}(\lambda)\geq N^{-1}(\sigma-\sup_\Omega h(x,0))>0.
\end{equation}

Note that $N$ and $\sigma$ are independent of the choice of $P$ and $u$, then by \eqref{e6.8} and \eqref{e6.14}, we obtain
\begin{equation}\label{e6.15}
\sup_{\Omega}u\leq \max\{c_1\sup_{\Omega}\frac{(h(x,0)-F[\alpha^{i\bar{p}}\chi_{j\bar{p}}])}{\underset{\min}{\lambda}(\alpha^{a\bar{l}}\varphi_{b\bar{l}})},0\},
\end{equation}
In the following $c_i$ will denote a positive constant depending only on $\alpha$, $\chi$, $F$, $h$ and $n$. Note that $\underset{\min}{\lambda}(\alpha^{a\bar{l}}\varphi_{b\bar{l}})$ is independent of the choice of local coordinate system.

Next, we want to estimate $\inf_{\Omega}u$. We divided into two cases to consider: (i) $\inf_{\Omega}u=-\limsup_{z\rightarrow \infty}|u|$, in this case, \eqref{e6.4} obviously holds; (ii) $\inf_{\Omega}u<-\limsup_{z\rightarrow \partial\Omega\cup \infty}|u|$, then there exists some $Q\in \Omega$ such that $u(Q)=\inf_{\Omega}u$. By maximum principle, we have at $Q$ that
\begin{equation}\label{e6.16}
\partial u=0,\;\;\;
\sqrt{-1}\partial\bar{\partial} u\geq 0.
\end{equation}
By \eqref{e6.3}, \eqref{e6.16} and the fact that \eqref{e6.3} is an elliptic equation, we get
\begin{equation}\label{e6.17}
\begin{aligned}
h(x,\varphi u-u)&=F[\alpha^{i\bar{p}}(\chi+\sqrt{-1}\partial\bar{\partial}(\varphi u-u))_{j\bar{p}}]\\
&=F[\alpha^{i\bar{p}}(\chi+\sqrt{-1}u\partial\bar{\partial}\varphi+2{\rm Re}(\sqrt{-1}\partial u\wedge \bar{\partial}\varphi)
+\sqrt{-1}(\varphi-1)\partial\bar{\partial} u)_{j\bar{p}}]\\
&\leq F[\alpha^{i\bar{p}}(\chi+\sqrt{-1}u\partial\bar{\partial}\varphi)_{j\bar{p}}].
\end{aligned}
\end{equation}
Hence $F(A)=f(\lambda_1,\ldots,\lambda_n)$ is concave and $\lambda[\alpha^{i\bar{p}}\chi_{j\bar{p}}]\in \Gamma$, we have
\begin{equation}\label{e6.18}
F[\alpha^{i\bar{p}}(\chi+\sqrt{-1}u\partial\bar{\partial}\varphi)_{j\bar{p}}]-F[\alpha^{i\bar{p}}\chi_{j\bar{p}}]
\leq F^{ab}[\alpha^{i\bar{p}}\chi_{j\bar{p}}](\alpha^{a\bar{l}}u\varphi_{b\bar{l}}).
\end{equation}
By condition \textbf{(a5)}, we obtain
\begin{equation}\label{en6.22}
h(x,\varphi u-u)\geq h(x,0).
\end{equation}
By \eqref{e6.17}, \eqref{e6.18} and \eqref{en6.22}, we obtain
\begin{equation}
F^{ab}[\alpha^{i\bar{p}}\chi_{j\bar{p}}](\alpha^{a\bar{l}}\varphi_{b\bar{l}})u\geq h(x,0)-F[\alpha^{i\bar{p}}\chi_{j\bar{p}}],
\end{equation}
choose a new coordinate system at $Q$ such that $[\alpha^{i\bar{p}}\chi_{j\bar{p}}]$ is a diagonal matrix, and note that $F^{kk}[\alpha^{i\bar{p}}\chi_{j\bar{p}}]>0$ for $k=1,\cdots,n$, then
\begin{equation}\label{e6.20}
\inf_{\Omega}u\geq \min\{c_2\inf_{\Omega}\frac{(h(x,0)-F[\alpha^{i\bar{p}}\chi_{j\bar{p}}])}{\underset{\min}{\lambda}(\alpha^{a\bar{l}}\varphi_{b\bar{l}})},0\}.
\end{equation}

Finally, by \eqref{e6.15} and \eqref{e6.20}, we obtain
\begin{equation}
\sup_{\Omega}|u|\leq c_3\sup_{\Omega}\frac{|h(x,0)-F[\alpha^{i\bar{p}}\chi_{j\bar{p}}]|}{\underset{\min}{\lambda}(\alpha^{a\bar{l}}\varphi_{b\bar{l}})}+\limsup_{z\rightarrow \infty}|u|.
\end{equation}
\end{proof}

Now, we can proof Theorem \ref{t6.1}.
\begin{proof}[Proof of Theorem \ref{t6.1}]
Let $u$ satisfies \eqref{e6.1}, then define
\begin{equation}\label{e6.22}
v=\frac{u}{\varphi-1}.
\end{equation}
We can verify that $v$ satisfies equation \eqref{e6.3}. By Lemma \ref{l6.2}, we obtain that
\begin{equation}\label{e6.23}
\sup_{\Omega}|v|\leq \tilde{c}\sup_{\Omega}\frac{|h(x,0)-F[\alpha^{i\bar{p}}\chi_{j\bar{p}}]|}{\underset{\min}{\lambda}(\alpha^{a\bar{l}}\varphi_{b\bar{l}})}+\limsup_{z\rightarrow \infty}|v|,
\end{equation}
where $\tilde{c}$ is a positive constant depending only on $\alpha$, $\chi$, $F$, $h$ and $n$. Hence by \eqref{e6.22} and \eqref{e6.23}, we have
\begin{equation}
\begin{aligned}
\sup_{\Omega}|u|&\leq \sup_{\Omega}|v|\sup_{\Omega}|\varphi-1|\\
&\leq\Big(\tilde{c}\sup_{\Omega}\frac{|h(x,0)-F[\alpha^{i\bar{p}}\chi_{j\bar{p}}]|}{\underset{\min}{\lambda}(\alpha^{a\bar{l}}\varphi_{b\bar{l}})}+(\limsup_{z\rightarrow \infty}|u|)(\sup_\Omega |(\varphi-1)^{-1}|)\Big)\sup_{\Omega}|\varphi-1|\\
&\leq c(1+\limsup_{z\rightarrow \infty}|u|),
\end{aligned}
\end{equation}
where $\tilde{c}$ is a positive constant depending only on $\Omega$, $\alpha$, $\chi$, $F$, $h$, $\varphi$ and $n$. Then we can prove this lemma.
\end{proof}

Note that there is a condition \eqref{en6.2} in Theorem \ref{t6.1}:
\[\sup_{M}\frac{|h(x,0)-F[\alpha^{i\bar{p}}\chi_{j\bar{p}}]|}{\underset{\min}{\lambda}(\alpha^{a\bar{l}}\varphi_{b\bar{l}})}<\infty.\]
In particular, we shall investigate what form the above condition take on pseudoconvex domains of the form appearing in Proposition \ref{pb2.3} or \ref{pb2.4}.

Let $\Omega$ be a pseudoconvex domain on a Hermitian manifold $(M,h_{i\bar{j}})$ as in Proposition \ref{pb2.3} or \ref{pb2.4} with $l=0$, $\varphi$ is the defining function of $\Omega$, let $\alpha_{i\bar{j}}=-(\log(-\varphi))_{i\bar{j}}+h_{i\bar{j}}$ be a Hermitian metric on $\Omega$. Then we obtain
\begin{equation}\label{e6.26}
\begin{aligned}
\underset{\min}{\lambda}(\alpha^{a\bar{l}}\varphi_{b\bar{l}})&\geq \underset{\min}{\lambda}(\alpha^{i\bar{j}})\underset{\min}{\lambda}(\varphi_{i\bar{j}})\\
&=\frac{\underset{\min}{\lambda}(\varphi_{i\bar{j}})}{\underset{\max}{\lambda}(\alpha_{i\bar{j}})}\\
&=\frac{\underset{\min}{\lambda}(\varphi_{i\bar{j}})}{\underset{\max}{\lambda}(\frac{\varphi_{i\bar{j}}}{-\varphi}+\frac{\varphi_i\varphi_{\bar{j}}}{\varphi^2}+h_{i\bar{j}})}\\
&\geq \hat{c}(n)\frac{\varphi^2\underset{\min}{\lambda}(h^{a\bar{j}}\varphi_{b\bar{j}})}{-\varphi |\partial\bar{\partial}\varphi|_h+|\partial \varphi|^2_h+\varphi^2},
\end{aligned}
\end{equation}
where $\hat{c}$ is a positive constant depends only on $n$. Then we obtain the a priori $C^{0}$-estimates on such domains as follows:
\begin{corollary}
Let $\Omega$ be a domain in a complete Hermitian manifold $(M,h_{i\bar{j}})$ as in Proposition \ref{pb2.3} or \ref{pb2.4} with $l=0$, $\varphi$, $|\partial \varphi|_h$ and $|\partial\bar{\partial}\varphi|_h$ are bounded on $\Omega$. Define $g=-\log(-\varphi)$ and
\[\alpha=\sqrt{-1}\sum_{i,j=1}^{n}(g_{i\bar{j}}+h_{i\bar{j}})dz_i\wedge dz_{\bar{j}}.\]
Fix a real $(1,1)$-form $\chi$ on $(\Omega,\alpha)$ such that $\lambda[\alpha^{i\bar{p}}\chi_{j\bar{p}}]\in \Gamma$, given $h\in C^0(\Omega\times\mathbb{R})$,
let $u\in C^2(\Omega)$ satisfy that
\begin{equation}
F(\alpha^{i\bar{p}}(\chi_{j\bar{p}}+u_{j\bar{p}}))=h(x,u)
\end{equation}
Suppose that the above equation satisfies conditions \emph{\textbf{(a1)}}-\emph{\textbf{(a5)}}, and assume that
\begin{equation}\label{e6.28}
\sup_{\Omega}\frac{|h(x,0)-F[\alpha^{i\bar{p}}\chi_{j\bar{p}}]|}{\varphi^2\underset{\min}{\lambda}(h^{a\bar{j}}\varphi_{b\bar{j}})}<\infty,
\end{equation}
then there is a constant $c$ depending only on $\Omega$, $\alpha$, $\chi$, $F$, $h$, $\varphi$ and $n$ such that
\begin{equation}
|u|\leq c(1+\limsup_{z\rightarrow \infty}|u|).
\end{equation}
\end{corollary}
\begin{proof}
By Theorem \ref{t6.1} and \eqref{e6.26}, we can prove this corollary.
\end{proof}
\noindent \textbf{Remark}: Since $\underset{\min}{\lambda}(h^{a\bar{j}}\varphi_{b\bar{j}})$ is bounded, condition \eqref{e6.28} indicates that $|h(x,0)-F[\alpha^{i\bar{p}}\chi_{j\bar{p}}]|$ tends to 0 as $x$ approaches $\partial \Omega$, and the decay rate of $|h(x,0)-F[\alpha^{i\bar{p}}\chi_{j\bar{p}}]|$ near the boundary $\partial \Omega$ is at least $\varphi^2$.

If, in addition, assuming that $h_u(x,u)$ is strictly positive, we can obtain the following:
\begin{theorem}\label{t6.4}
Let $(\Omega,\alpha)$ be a complete Hermitian manifold with bounded geometry of order $l$, $l\geq 0$. Fix a real $(1,1)$-form $\chi$ on $(\Omega,\alpha)$ such that $\lambda[\alpha^{i\bar{p}}\chi_{j\bar{p}}]\in \Gamma$, given $h\in C^1(\Omega \times\mathbb{R})$ such that $h_u\geq a>0$ for some constant $a$,
let $u\in C^2(\Omega)$ be a bounded function satisfying
\begin{equation}
F(\alpha^{i\bar{p}}(\chi_{j\bar{p}}+u_{j\bar{p}}))=h(x,u)
\end{equation}
Suppose that the above equation satisfies conditions \emph{\textbf{(a1)}}-\emph{\textbf{(a5)}},
then
\begin{equation}
\sup_\Omega |u|\leq \sup_\Omega\frac{|F(\alpha^{i\bar{p}}\chi_{j\bar{p}})-h(x,0)|}{a}.
\end{equation}
\end{theorem}
\begin{proof}
First, we want to estimate $\sup_{\Omega}u$, we may assume that $\sup_{\Omega}u>0$. By generalized maximum principle (see Proposition \ref{pb4.2}), there exists a sequence of points $\{x_k\}$ in $\Omega$ such that $\lim u(x_k)=\sup u$ and $\overline{\lim}(u_{p\bar{q}}(x_k))\leq 0$. Then, at each $x_k$ such that $u(x_k)>0$ we obtain
\begin{equation}\label{e6.36}
F(\alpha^{i\bar{p}}(\chi_{j\bar{p}}+u_{j\bar{p}}))=h(x,u)\geq h(x,0)+au.
\end{equation}
Let $k\rightarrow \infty$, we see immediately that
\begin{equation}
\sup_\Omega u\leq \sup_\Omega\frac{F(\alpha^{i\bar{p}}\chi_{j\bar{p}})-h(x,0)}{a}.
\end{equation}
The estimate on $\inf_\Omega u$ is carried out similarly and then \eqref{e6.36} becomes obvious.
\end{proof}

\section{Proof of Theorem 1.4}

In this section, we consider the $C^{1}$ and $C^{2}$ estimates of \eqref{e1.1} on Hermitian manifolds with bounded geometry. Since a bounded function on a noncompact manifold may not attain its extremum at some point, the key idea of this section is to apply the generalized maximum principle (see \eqref{eb4.37}, \eqref{eb4.38} and \eqref{eb4.39}), rather than the classical maximum principle, to establish the $C^{2}$ estimates. We use some of the language and approaches of Sz\'{e}kelyhidi-Tosatti-Weinkove \cite{STW17}.

Let $(M,\alpha)$ be a Hermitian manifold with bounded geometry of order 2 of complex dimension $n$ and write
\begin{equation}
\alpha=\sqrt{-1}\alpha_{i\bar{j}}dz^i\wedge d\bar{z}^j>0.
\end{equation}
Fix a background $(1,1)$-form $\chi=\sqrt{-1}\chi_{i\bar{j}}dz^i\wedge d\bar{z}^j$ which is not necessarily positive definite. For $u:M\rightarrow \mathbb{R}$, define a new tensor $g_{i\bar{j}}$ by
\begin{equation}\label{e4.3}
g_{i\bar{j}}:=\chi_{i\bar{j}}+u_{i\bar{j}}.
\end{equation}
And we can define the endomorphism of $T^{1,0}M$ given by $A_j^i=\alpha^{i\bar{p}}g_{j\bar{p}}$.

In the following, we consider the a priori estimates of the equations for $u\in C^4(M)\cap \{u\in C^2(M)|\;|u|_2<\infty\}$ as follows:
\begin{equation}\label{en4.3}
F(A)=h,
\end{equation}
for a given function $h\in \hat{C}^2(M)$ on $M$, where
\begin{equation}
F(A)=f(\lambda_1,\ldots,\lambda_n)
\end{equation}
is a smooth symmetric function of the eigenvalues of $A$, and the equation \eqref{en4.3} satisfies conditions \textbf{(a1)}-\textbf{(a4)} and $\chi$ has bounded geometry of order 2 with respect to $(M,\alpha)$.

Note that, if $\underline{u}$ is a $\mathcal{C}$-subsolution and $\underline{u}\in \tilde{C}^{\infty}(M)$, then replacing $\chi$ by
\begin{equation}
\chi'_{i\bar{j}}=\chi_{i\bar{j}}+\underline{u}_{i\bar{j}},
\end{equation}
it is easy to see $\chi'$ also has bounded geometry of order 2. We may assume that $\underline{u}=0$. The important consequence of $0$ being a $\mathcal{C}$-subsolution is the Proposition 2.3 of \cite{STW17}.

Our main goal is the following estimate:
\begin{equation}\label{eb4.17}
\sup_M|\sqrt{-1}\partial\bar{\partial}u|_\alpha\leq C\biggl(\sup_M|\nabla u|_\alpha^2+1\biggl).
\end{equation}
In fact, the estimate \eqref{eb4.17} is equivalent to the bound
\begin{equation}
\lambda_1\leq CK,
\end{equation}
where $K=\sup_M|\nabla u|_\alpha^2+1$ and $\lambda_1$ is the largest eigenvalue of $A=(A_j^i)=(\alpha^{i\bar{p}}g_{j\bar{p}})$. Indeed, our assumption on the cone $\Gamma$ implies that $\sum_i \lambda_i>0$ (see \cite{STW17}). Then, if $\lambda_1$ is bounded from above by $CK$, so is $|\lambda_i|$ for all $i$, giving the same bound for $\sup_M|\sqrt{-1}\partial\bar{\partial}u|_\alpha$.

The generalized maximum principle on non-compact manifolds (see Proposition \ref{pb4.2}) are our main tools in this section.

To use the generalized maximum principle, in the sequel, we assume that $u\in C^4(M)\cap \{u\in C^2(M)|\;|u|_2<\infty\}$. We consider the function
\begin{equation}
H=\log \lambda_1+\phi(|\nabla u|_\alpha^2)+\psi(u),
\end{equation}
where $\phi$ is defined by
\begin{equation}
\phi(t)=-\frac{1}{2}\log \biggl(1-\frac{t}{2K}\biggl),
\end{equation}
so that $\phi(|\nabla u|_\alpha^2)\in [0,\frac{1}{2}\log 2]$ satisfies
\begin{equation}
\frac{1}{4K}<\phi'<\frac{1}{2K},\;\;\;\;\phi''=2(\phi')^2,
\end{equation}
and $\psi$ is defined by
\begin{equation}
\psi(t)=D_1e^{-D_2 t},
\end{equation}
for sufficiently large uniform constants $D_1,D_2>0$ to be chosen later. By the $C^0$ bound on $u$, the quantity $\psi(u)$ is uniformly bounded.


The quantity $H$ need not be smooth at any point, because the largest eigenvalue of $A$ may have eigenspace of dimension larger than 1. Therefore, we can not apply Proposition \ref{pb4.2} to $H$ directly. Since $u\in C^4(M)\cap \{u\in C^2(M)|\;|u|_2<\infty\}$, $H$ is bounded from above in $M$, there exists a a sequence $\{p_i\}$ in $M$ such that $\lim H(p_i)=\sup H$. If $H$ achieves its maximum at some point $p\in M$, then we can use the same method as in \cite{STW17} to obtain the estimate \eqref{eb4.17}. Hence in the following, we assume that
\begin{equation}\label{ne4.13}
H<\sup H \;\;{\rm on} \;\;M.
\end{equation}

For some $i$, we work at the point $p_i$, where $i$ is a sufficiently large positive integer to be chosen later. In orthonormal complex coordinates for $\alpha$ centered at this point, such that $g$ is diagonal and $\lambda_1=g_{1\bar{1}}$. To take care of this, we carry out a perturbation as in \cite{S18}, choosing local coordinates such that $H$ achieves $H(p_i)$ at the origin, where $A$ is diagonal with eigenvalues $\lambda_1\geq \ldots \geq\lambda_n$. We fix a diagonal matrix $B$ independent of $p_i$, with $B_1^1=0$ and $0<B_2^2<\ldots<B^n_n$, and we define $\tilde{A}=A-B$, denoting its eigenvalues by $\tilde{\lambda}_1,\ldots,\tilde{\lambda}_n$.

At the origin, we have
\begin{equation}
\tilde{\lambda}_1=\lambda_1\;\;\;\;{\rm and}\;\;\;\;\tilde{\lambda}_i=\lambda_i-B_i^i,\;\;i>1,
\end{equation}
and $\tilde{\lambda}_1>\tilde{\lambda}_2>\ldots>\tilde{\lambda}_n$. As discussed above, our assumption on the cone $\Gamma$ implies that $\sum_i \lambda_i>0$, and we fix the matrix $B$ small enough so that
\begin{equation}
\sum_i \tilde{\lambda}_i>-1.
\end{equation}
We can choose this matrix $B$ in such a way that, in addition,
\begin{equation}\label{eb4.25}
\sum_{p>1}\frac{1}{\lambda_1-\tilde{\lambda}_p}\leq C,
\end{equation}
for some fixed constant $C$ depending on the dimension $n$. Next, we will show that $\tilde{\lambda}_i$ is smooth on $B_r(0)$ for some $r>0$ independent of $p_i$. Define
\begin{equation}\label{eb4.26}
G(\lambda,\tilde{A}(x))=\det(\lambda I-\tilde{A}(x)).
\end{equation}
Obviously, by \eqref{eb4.25} and \eqref{eb4.26}, we have
\begin{equation}\label{eb4.27}
\begin{aligned}
\frac{\partial G(\lambda,\tilde{A}(x))}{\partial \lambda}\biggl|_{x=0,\lambda=\tilde{\lambda}_1(0)}&=\prod_{i=2}^{n}({\lambda}_1(0)-\tilde{\lambda}_i(0))\\
&\geq \biggl(\frac{n-1}{\sum_{p>1}\frac{1}{\lambda_1(0)-\tilde{\lambda}_p(0)}}\biggl)^{n-1}\\
&\geq \biggl(\frac{n-1}{C}\biggl)^{n-1}.
\end{aligned}
\end{equation}
By implicit function theorem, there exists a smooth function $\lambda=\lambda^1(\tilde{A})$ such that
\begin{equation}\label{eb4.28}
G(\lambda^1(\tilde{A}),\tilde{A})=0\;\;\;\;{\rm and }\;\;\;\;\lambda^1(\tilde{A}(0))=\lambda_1(0)
\end{equation}
on a neighborhood $B_{r'}(A(0))$ of $A(0)$($A$ can be seen as a vector on $\mathbb{C}^{n\times n}$). Since the matrix $A$ depends only on $\chi_{i\bar{j}}$, $u_{i\bar{j}}$ and $\alpha_{i\bar{j}}$, we can assume that $r'$ is a positive constant independent of $p_i$ and on $B_{r''}(0)\subseteq M$ we also have \eqref{eb4.28} for some $r''>0$ independent of $p_i$. Similarly, we obtain $\lambda=\lambda^i(\tilde{A})$ such that
\begin{equation}\label{eb4.29}
G(\lambda^i(\tilde{A}),\tilde{A})=0\;\;\;\;{\rm and }\;\;\;\;\lambda^i(\tilde{A}(0))=\tilde{\lambda}_i(0)
\end{equation}
on $B_{r''}(0)\subseteq M$. Differentiating the \eqref{eb4.29}, we get
\begin{equation}
dG(\lambda^i(\tilde{A}),\tilde{A})=\frac{\partial G}{\partial \lambda^i}d\lambda^i+\sum_{p,q=1}^n\frac{\partial G}{\partial A_p^q}dA_p^q=0.
\end{equation}
So $d\lambda^i$ can be small enough if $r''\rightarrow 0$, we have a constant $r>0$ independent of $p_i$ such that
\begin{equation}
\lambda^1(\tilde{A}(x))>\lambda^2(\tilde{A}(x))>\ldots>\lambda^n(\tilde{A}(x))
\end{equation}
on $B_r(0)$, so $\tilde{\lambda}_1=\lambda^1(\tilde{A}(x))$ is smooth on $B_r(0)$.

Now, after possibly shrinking the chart, the quantity
\begin{equation}
\widetilde{H}=\log \tilde{\lambda}_1+\phi(|\nabla u|_\alpha^2)+\psi(u)
\end{equation}
is smooth on the chart $B_r(p_i)$, in the sequel, we compute in this chart. Let $\sup H=L$, by \eqref{ne4.13}, we have $L-\widetilde{H}>0$ on $M$. Let $\{p_i\}$ be a sequence in $M$ as above such that $\lim H(p_i)=\sup H$. Since $(M,\alpha)$ has bounded geometry of order 2, then there is a non-negative function $\varphi^{p_i}:B_r(p_i)\rightarrow \mathbb{R}$ such that {\rm (i)} $\varphi^{p_i}_{(p_i)}=1$ and $\varphi^{p_i}=0$ on $\partial B_r(p_i)$, {\rm (ii)} $\varphi^{p_i}\leq c$, $|\nabla \varphi^{p_i}|\leq c$ and $(\varphi_{i\bar{j}}^{p_i})\geq -c(\alpha_{i\bar{j}})$, where $c$ is a positive constant independent of $i$. Consider $(L-\widetilde{H})/\varphi^{p_i}$ as a function defined on $B_r(p_i)$. Let $x_i\in B_r(p_i)$ be such that
\begin{equation}
\biggl(\frac{L-\widetilde{H}}{\varphi^{p_i}}\biggl)(x_i)=\inf_{B_r(p_i)}\biggl(\frac{L-\widetilde{H}}{\varphi^{p_i}}\biggl).
\end{equation}
Then,
\begin{equation}
\frac{L-\widetilde{H}}{\varphi^{p_i}}(x_i)\leq \frac{L-\widetilde{H}}{\varphi^{p_i}}(p_i)=L-H(p_i),
\end{equation}
\begin{equation}
\frac{d(L-\widetilde{H})}{L-\widetilde{H}}(x_i)=\frac{d\varphi^{p_i}}{\varphi^{p_i}}(x_i),
\end{equation}
\begin{equation}
\biggl(\frac{(L-\widetilde{H})_{p\bar{q}}}{L-\widetilde{H}}(x_i)\biggl)\geq \biggl(\frac{\varphi^{p_i}_{p\bar{q}}}{\varphi^{p_i}}(x_i)\biggl).
\end{equation}
Using the property of $\varphi^{p_i}$ we see that
\begin{equation}\label{eb4.37}
0<L-\widetilde{H}(x_i)\leq c(L-H(p_i)),
\end{equation}
\begin{equation}\label{eb4.38}
|d\widetilde{H}(x_i)|\leq c(L-H(p_i)),
\end{equation}
\begin{equation}\label{eb4.39}
(\widetilde{H}_{p\bar{q}}(x_i))\leq c(L-H(p_i))(\alpha_{p\bar{q}})
\end{equation}

We will apply \eqref{eb4.37}, \eqref{eb4.38} and \eqref{eb4.39} to $\widetilde{H}$ instead of the classical maximum principles. Our goal is to obtain the bounded $\tilde{\lambda}_1\leq CK$ at $x_i$, which will give us the required estimate \eqref{eb4.17}. Hence, we may and do assume that $\tilde{\lambda}_1\gg K$ at this point. We now differentiate $\widetilde{H}$ at $x_i$ and, we use subscripts $k$ and $\bar{l}$ to denote the partial derivatives $\partial/\partial z^k$ and $\partial/\partial \bar{z}^l$. We have
\begin{equation}\label{eb4.40}
\begin{aligned}
\widetilde{H}_k&=\frac{\tilde{\lambda}_{1,k}}{\tilde{\lambda}_1}+\phi'(\alpha^{p\bar{q}}u_p u_{\bar{q}k}+\alpha^{p\bar{q}}u_{pk}u_{\bar{q}}+(\alpha^{p\bar{q}})_k u_p u_{\bar{q}})+\psi'u_k\\
&=\frac{\tilde{\lambda}_{1,k}}{\tilde{\lambda}_1}+\phi'(u_p u_{\bar{p}k}+u_{pk}u_{\bar{p}}+(\alpha^{p\bar{q}})_k u_p u_{\bar{q}})+\psi'u_k\\
&=\frac{\tilde{\lambda}_{1,k}}{\tilde{\lambda}_1}+\phi'V_k+\psi' u_k,
\end{aligned}
\end{equation}
where $V_k:=u_p u_{\bar{p}k}+u_{pk}u_{\bar{p}}+(\alpha^{p\bar{q}})_k u_p u_{\bar{q}}$.

Next, we will only make a small change to the proof of Theorem 2.2 in \cite{STW17}, if $i$ is large enough, by (3.28) of \cite{STW17} and \eqref{eb4.39}, we have
\begin{equation}\label{eb4.41}
\begin{aligned}
0\geq -&\frac{F^{pq,rs}\nabla_1 g_{q\bar{p}}\nabla_{\bar{1}}g_{s\bar{r}}}{\lambda_1}-\frac{F^{kk}|\tilde{\lambda}_{1,k}|^2}{\lambda_1^2}\\
&+\sum_p \frac{F^{kk}}{6K}(|u_{pk}|^2+|u_{p\bar{k}}|^2)+\phi''F^{kk}|V_k|^2\\
&+\psi''F^{kk}|u_k|^2+\psi'F^{kk}u_{k\bar{k}}-C(F^{kk}\lambda_1^{-1}|g_{1\bar{1}k}|+\mathcal{F}).
\end{aligned}
\end{equation}
where $\nabla$ is the Chern connection of $\alpha$, and $\mathcal{F}=\sum_k F^{kk}$.

We now deal with two cases separately, depending on a small constant $\delta=\delta_{D_1,D_2}>0$ to be determined shortly, and which will depend on the constants $D_1$ and $D_2$.

\emph{Case} 1. Assume $\delta \lambda_1\geq -\lambda_n$. Define the set
\begin{equation}
I=\{i:F^{ii}>\delta^{-1}F^{11}\}.
\end{equation}
From \eqref{eb4.40} and Cauchy-Schwarz inequality, we get
\begin{equation}\label{eb4.43}
\begin{aligned}
-\sum_{k\notin I}\frac{F^{kk}|\tilde{\lambda}_{1,k}|^2}{\lambda_1^2}&=-\sum_{k\notin I}F^{kk}|\phi'V_k+\psi'u_k-\widetilde{H}_k|^2\\
&=-(\phi')^2\sum_{k\notin I}F^{kk}|V_k|^2-(\psi')^2\sum_{k\notin I}F^{kk}|u_k|^2-\sum_{k\notin I}F^{kk}|\widetilde{H}_k|^2\\
&\;\;\;\;-2\sum_{k\notin I}F^{kk}{\rm Re}(\phi'V_k\overline{\psi'u_k})+2\sum_{k\notin I}F^{kk}{\rm Re}(\widetilde{H}_k\overline{\psi'u_k})\\
&\;\;\;\;+2\sum_{k\notin I}F^{kk}{\rm Re}(\widetilde{H}_k\overline{\phi'V_k)}\\
&\geq-(\phi')^2\sum_{k\notin I}F^{kk}|V_k|^2-(\psi')^2\sum_{k\notin I}F^{kk}|u_k|^2-\sum_{k\notin I}F^{kk}|\widetilde{H}_k|^2\\
&\;\;\;\;-\frac{2}{3}(\phi')^2\sum_{k\notin I}F^{kk}|V_k|^2-\frac{3}{2}(\psi')^2\sum_{k\notin I}F^{kk}|u_k|^2\\
&\;\;\;\;-\frac{1}{3}(\phi')^2\sum_{k\notin I}F^{kk}|V_k|^2
-\frac{1}{2}(\psi')^2\sum_{k\notin I}F^{kk}|u_k|^2-C\sum_{k\notin I}F^{kk}|\widetilde{H}_k|^2\\
&\geq -\phi''\sum_{k\notin I}F^{kk}|V_k|^2-3(\psi')^2\delta^{-1}F^{11}K-C\sum_{k\notin I}F^{kk}|\widetilde{H}_k|^2.
\end{aligned}
\end{equation}

For $k\in I$ we have, in the same way,
\begin{equation}\label{eb4.44}
-2\delta\sum_{k\in I}\frac{F^{kk}|\tilde{\lambda}_{1,k}|^2}{\lambda_1^2}\geq -2\delta\phi''\sum_{k\in I}F^{kk}|V_k|^2-6\delta(\psi')^2\sum_{k\in I}F^{kk}|u_k|^2-C\sum_{k\in I}F^{kk}|\widetilde{H}_k|^2.
\end{equation}
We wish to use some of the good $\psi''F^{kk}|u_k|^2$ term in \eqref{eb4.41} to control the last term in \eqref{eb4.44}. For this, we assume that $\delta$ is chosen so small (depending on $\psi$, i.e. on $D_1, D_2$ and the maximum of $|u|$), such that
\begin{equation}\label{eb4.45}
6\delta(\psi')^2<\frac{1}{2}\psi''.
\end{equation}
Since $\psi''$ is strictly positive, such a $\delta>0$ exists.

Using this together with \eqref{eb4.43} and \eqref{eb4.44} in \eqref{eb4.41}, and we can choose some $i$ such that $|\widetilde{H}_k(x_i)|$ is sufficiently small by \eqref{eb4.38}, then we have
\begin{equation}
\begin{aligned}
0\geq -&\frac{F^{pq,rs}\nabla_1 g_{q\bar{p}}\nabla_{\bar{1}}g_{s\bar{r}}}{\lambda_1}-(1-2\delta)\sum_{k\in I}\frac{F^{kk}|\tilde{\lambda}_{1,k}|^2}{\lambda_1^2}\\
&+\sum_p \frac{F^{kk}}{6K}(|u_{pk}|^2+|u_{p\bar{k}}|^2)+\frac{1}{2}\psi''F^{kk}|u_k|^2+\psi'F^{kk}u_{k\bar{k}}\\
&-3(\psi')^2\delta^{-1}F^{11}K-C(F^{kk}\lambda_1^{-1}|g_{1\bar{1}k}|+\mathcal{F}).
\end{aligned}
\end{equation}

In the sequel, we can only change the term $-2(\psi')^2\delta^{-1}F^{11}K$ in the proof of case 1 in \cite{STW17} to $-3(\psi')^2\delta^{-1}F^{11}K$. By (3.55) of \cite{STW17}, we have
\begin{equation}\label{eb4.47}
\begin{aligned}
0\geq &F^{11}\biggl(\frac{\lambda_1^2}{40K}-3(\psi')^2\delta^{-1}K\biggl)+\biggl(\frac{1}{2}\psi''+C_\varepsilon\psi'\biggl)F^{kk}|u_k|^2\\
&-C_0\mathcal{F}+\varepsilon C_0\psi'\mathcal{F}-\psi'F^{kk}(\chi_{k\bar{k}}-g_{k\bar{k}}).
\end{aligned}
\end{equation}
for a uniform $C_0$. Under the assumption that the function $\underline{u}=0$ in a $\mathcal{C}$-subsolution, and that $\lambda_1\gg 1$, we may apply Proposition 2.3 of \cite{STW17} and see that there is a uniform positive number $\kappa>0$ such that one of two possibilities occurs:

(a) We have $F^{kk}(\chi_{k\bar{k}}-g_{k\bar{k}})>\kappa \mathcal{F}$. In this case we have
\begin{equation}
0\geq F^{11}\biggl(\frac{\lambda_1^2}{40K}-3(\psi')^2\delta^{-1}K\biggl)+\biggl(\frac{1}{2}\psi''+C_\varepsilon\psi'\biggl)F^{kk}|u_k|^2-C_0\mathcal{F}+\varepsilon C_0\psi'\mathcal{F}-\psi'\kappa \mathcal{F}.
\end{equation}
We first choose $\varepsilon>0$ such that $\varepsilon C_0<\frac{1}{2}\kappa$. We then choose the parameter $D_2$ in the definition of $\psi(t)=D_1e^{-D_2 t}$ to be large enough so that
\begin{equation}
\frac{1}{2}\psi''>C_\varepsilon |\psi'|.
\end{equation}
At this point, we have
\begin{equation}
0\geq F^{11}\biggl(\frac{\lambda_1^2}{40K}-3(\psi')^2\delta^{-1}K\biggl)-C_0\mathcal{F}-\frac{1}{2}\psi'\kappa \mathcal{F}.
\end{equation}
We now choose $D_1$ so large that $-\frac{1}{2}\psi'\kappa>C_0$, which implies
\begin{equation}
\frac{\lambda_1^2}{40K}\leq 3(\psi')^2\delta^{-1}K.
\end{equation}
Note that $\delta$ is determined by the choices of $D_1$ and $D_2$, according to \eqref{eb4.45}, so we obtain the required upper bound for $\lambda_1/K$.

(b) We have $F^{11}>\kappa \mathcal{F}$. With the choices of constants made above, \eqref{eb4.47} implies that
\begin{equation}
0\geq \kappa \mathcal{F}\biggl(\frac{\lambda_1^2}{40K}-3(\psi')^2\delta^{-1}K\biggl)
-C_0\mathcal{F}+\varepsilon C_0\psi'\mathcal{F}+C_1\psi'\mathcal{F}+\psi'F^{kk}g_{k\bar{k}}.
\end{equation}
for another uniform constant $C_1$. Since $F^{kk}g_{k\bar{k}}\leq \mathcal{F}\lambda_1$, we can divide through by $\mathcal{F}K$ and obtain
\begin{equation}
0\geq \frac{\kappa \lambda_1^2}{40K^2}-C_2(1+K^{-1}+\lambda_1K^{-1}),
\end{equation}
for a uniform $C_2$. The required upper bound for $\lambda_1/K$ follows from this.

\emph{Case} 2. We now assume that $\delta\lambda_1<-\lambda_n$, with all the constants $D_1$, $D_2$ and $\delta$ fixed as in the previous case. By (3.56) of \cite{STW17}, we have
\begin{equation}\label{eb4.54}
0\geq -\frac{3}{2}\frac{F^{kk}|\tilde{\lambda}_{1,k}|^2}{\lambda_1^2}+\frac{\delta^2}{10nK}\mathcal{F}\lambda_1^2+F^{kk}\phi''|V_k|^2-C\mathcal{F}\lambda_1.
\end{equation}
Similarly to \eqref{eb4.43}, we obtain
\begin{equation}
\frac{3}{2}\frac{F^{kk}|\tilde{\lambda}_{1,k}|^2}{\lambda_1^2}\leq F^{kk}\phi''|V_k|^2+C\mathcal{F}K.
\end{equation}
Returning to \eqref{eb4.54}, we obtain, since we may assume $\lambda_1\geq K$,
\begin{equation}
0\geq \frac{\delta^2\lambda_1^2}{10nK}\mathcal{F}-C\lambda_1\mathcal{F}.
\end{equation}
Dividing by $\lambda_1\mathcal{F}$ gives the required bound for $\lambda_1/K$.

Then, we immediately deduce the bound \eqref{eb4.17}, namely
\begin{equation}
\sup_M|\sqrt{-1}\partial\bar{\partial}u|_\alpha\leq C\biggl(\sup_M|\nabla u|_\alpha^2+1\biggl).
\end{equation}
A blow-up argument (Chapter 6 in \cite{S18}) combined with a Liouville theorem (Chapter 5 in \cite{S18}), shows that $\sup_M |\nabla u|_\alpha^2\leq C$, and so we get a uniform bound $|\Delta u|\leq C$.

By rewriting equation \eqref{en4.3} in real coordinates, we can then apply the Evans-Krylov Theorem \cite{r8}(see also Theorem 17.14 of \cite{GT98}) and deduce a uniform bound
\begin{equation}
|u|_{2+\beta}\leq C\;\;\;\;{\rm on}\;\;(M,\alpha)
\end{equation}
for a uniform $0<\beta<1$, where the norm is defined in \eqref{eb2.18} with respect to $(M,\alpha)$. Differentiating the equation and applying a standard bootstrapping argument, we finally obtain uniform higher-order estimates. Then we can prove Theorem \ref{t4.4}.

\section{The solvability of $F(A)=h$}

In this section, we consider the solvability of \eqref{e1.1}. The outline of this section is as follows: We will construct a solution of $F(A)=h(x,u)$ by first solving the equation $F(A)=h(x,u)+\epsilon u$ for $\epsilon\in(0,1)$. We will show that the a priori estimates for this equation are independent of $\epsilon$. Then, by letting $\epsilon\rightarrow 0$, we prove that the solution of $F(A)=h(x,u)+\epsilon u$ converges to a solution of $F(A)=h(x,u)$. This method is analogous to the approach of Tian-Yau \cite{TY90} in solving a Monge-Amp\`{e}re equation on noncompact K\"{a}hler manifolds.

First, we consider the behavior of the solution of \eqref{e1.1} when $x\rightarrow \infty$, similar to Proposition 6.1 of \cite{CY80}.
\begin{lemma}\label{l7.1}
Let $\Omega$ be a complete K\"{a}hler manifold with a bounded global K\"{a}hler potential $\varphi<0$, let $\alpha$ be a Hermitian metric with bounded geometry on $\Omega$.
Fix a real $(1,1)$-form $\chi$ on $(\Omega,\alpha)$ such that $\lambda[\alpha^{i\bar{p}}\chi_{j\bar{p}}]\in \Gamma$, given $h\in C^0(\Omega\times\mathbb{R})$,
a bounded function $u\in C^2(\Omega)$ satisfy that
\begin{equation}
F(\alpha^{i\bar{p}}(\chi_{j\bar{p}}+u_{j\bar{p}}))=h(x,u)
\end{equation}
Suppose that the above equation satisfies conditions \emph{\textbf{(a1)}}-\emph{\textbf{(a5)}}, and $h_u\geq \epsilon>0$ for some constant $\epsilon$. Assume that
\begin{equation}\label{e7.2}
\sup_{\Omega}\frac{|h(x,0)-F[\alpha^{i\bar{p}}\chi_{j\bar{p}}]|}{\underset{\min}{\lambda}(\alpha^{a\bar{l}}\varphi_{b\bar{l}})}<\infty,
\end{equation}
\begin{equation}
\frac{1}{c}(\alpha_{j\bar{p}})<(\chi_{j\bar{p}}+u_{j\bar{p}})<c(\alpha_{j\bar{p}}),
\end{equation}
where $c>0$ is a constant. Then, $|u|=O(|\varphi|)$.
\end{lemma}
\begin{proof}
Consider
\begin{equation}\label{e7.4}
\begin{aligned}
h(x,u)-F[\alpha^{i\bar{p}}\chi_{j\bar{p}}]&=\int_0^1 \frac{d}{dt}F(\alpha^{i\bar{p}}(\chi_{j\bar{p}}+tu_{j\bar{p}}))dt\\
&=\sum(\int_0^1F^{ab}(\alpha^{i\bar{p}}(\chi_{j\bar{p}}+tu_{j\bar{p}}))dt)\alpha^{a\bar{l}}u_{b\bar{l}},
\end{aligned}
\end{equation}
let
\begin{equation}\label{e7.5}
(A^{i\bar{j}})=(\int_0^1F^{ij}(\alpha^{i\bar{p}}(\chi_{j\bar{p}}+tu_{j\bar{p}}))dt).
\end{equation}
Then
\begin{equation}
\frac{t}{c}(\alpha_{i\bar{j}})+(1-t)(\chi_{i\bar{j}})\leq(\chi_{i\bar{j}}+tu_{i\bar{j}})\leq tc(\alpha_{i\bar{j}})+(1-t)(\chi_{i\bar{j}}),
\end{equation}
this means that
\begin{equation}\label{e7.6}
\{\lambda(\alpha^{i\bar{p}}(\chi_{j\bar{p}}+tu_{j\bar{p}})):t\in[0,1]\}\subset K\subset \Gamma_n,
\end{equation}
where $K$ is some bounded closed set in $\mathbb{R}^n$, and since $\sqrt{-1}\partial\bar{\partial}\varphi>0$, by \eqref{e7.5}, \eqref{e7.6} and the condition \textbf{(a1)}, there exists a constant $c_1>0$ such that
\begin{equation}\label{e7.7}
\frac{1}{c_1}\sum \delta_{ij}\varphi_{i\bar{j}}\leq\sum A^{i\bar{j}}\varphi_{i\bar{j}}\leq c_1\sum \delta_{ij}\varphi_{i\bar{j}}.
\end{equation}
Note that if $u>0$, we have
\begin{equation}\label{e7.8}
h(x,u)\geq h(x,0)+\epsilon u.
\end{equation}
By \eqref{e7.2}, \eqref{e7.4}, \eqref{e7.7} and \eqref{e7.8}, one can show that there exists a constant $c_2>0$ such that
\begin{equation}\label{e7.9}
\sum A^{i\bar{j}}\alpha^{i\bar{l}}(u+c_2\varphi)_{j\bar{l}}\geq \epsilon(u+c_2\varphi)
\end{equation}
on $\{x|u(x)>0\}$. The generalized maximum principle (see Proposition \ref{pb4.2}) can be applied to show that $u+c_2\varphi\leq 0$ in $\Omega$. This shows that $u\leq c_2(-\varphi)$. The same argument yields that $u\geq c_3\varphi$ for some constant $c_3$. Hence $|u|=O(|\varphi|)$ as desired.
\end{proof}

Suppose that $\lim_{x\rightarrow\infty}\varphi(x)=0$, we obtain the following estimates:
\begin{lemma}\label{t7.2}
Let $\Omega$ be a complete K\"{a}hler manifold with a bounded global K\"{a}hler potential $\varphi<0$ and $\lim_{x\rightarrow\infty}\varphi(x)=0$. Let $\alpha$ be a Hermitian metric on $\Omega$ with bounded geometry of order $k-1$, where $k\geq 4$.
Fix a real $(1,1)$-form $\chi$ with bounded geometry of order $k-1$ on $(\Omega,\alpha)$ such that $\lambda[\alpha^{i\bar{p}}\chi_{j\bar{p}}]\in \Gamma$, given $h\in \tilde{C}^{k-2+\beta}(\Omega \times\mathbb{R})$, where $\beta\in (0,1)$ is some uniform constant.
Let $u\in \tilde{C}^{k+\beta}(\Omega)$ satisfy that
\begin{equation}\label{e7.10}
F(\alpha^{i\bar{p}}(\chi_{j\bar{p}}+u_{j\bar{p}}))=h(x,u)+\epsilon u,
\end{equation}
\begin{equation}
\frac{1}{\hat{c}}(\alpha_{j\bar{p}})<(\chi_{j\bar{p}}+u_{j\bar{p}})<\hat{c}(\alpha_{j\bar{p}}),
\end{equation}
for some constant $\epsilon\in(0,1)$ and $\hat{c}>0$. Suppose that the above equation satisfies conditions \emph{\textbf{(a1)}}-\emph{\textbf{(a5)}}, and assume that
\begin{equation}
\sup_{\Omega}\frac{|h(x,0)-F[\alpha^{i\bar{p}}\chi_{j\bar{p}}]|}{\underset{\min}{\lambda}(\alpha^{a\bar{l}}\varphi_{b\bar{l}})}<\infty,
\end{equation}
and for all $i=1,\ldots,n$
\begin{equation}\label{e7.13}
\lim_{t\rightarrow\infty}f(\lambda[\alpha^{i\bar{p}}\chi_{j\bar{p}}]+t\textbf{\emph{e}}_i)=\infty
\end{equation}
$\textbf{\emph{e}}_i$ denotes the $i$th standard basis vector. Then, there is a constant $c$ depending only on $\Omega$, $\alpha$, $\chi$, $F$, $h$, $\varphi$, $k$ and $n$ such that
\begin{equation}
|u|_{k+\beta}\leq c.
\end{equation}
In particular, $c$ is independent of the choice of $\epsilon$.
\end{lemma}
\begin{proof}
Since $\lim_{x\rightarrow\infty}\varphi(x)=0$, by lemma \ref{l7.1}, we have
\begin{equation}\label{e7.15}
\limsup_{z\rightarrow \infty}|u|=0,
\end{equation}
by \eqref{e7.15} and Theorem \ref{t6.1}, we obtain a $C^0$ estimates:
\begin{equation}
|u|\leq \zeta,
\end{equation}
note that the constant $\zeta$ is independent of the choice of $\epsilon$. Then, we get
\begin{equation}\label{e7.17}
\inf_\Omega h(x,0)-\zeta(\sup_{\Omega \times\mathbb{R}} h_u+1)\leq h(x,u)+\epsilon u\leq \sup_\Omega h(x,0)+\zeta(\sup_{\Omega \times\mathbb{R}} h_u+1),
\end{equation}
by \eqref{e7.13} and \eqref{e7.17}, we obtain
\begin{equation}\label{e7.18}
\lim_{t\rightarrow\infty}f(\lambda[\alpha^{i\bar{p}}\chi_{j\bar{p}}]+t\textbf{e}_i)>\sup_\Omega h(x,0)+\zeta(\sup_{\Omega \times\mathbb{R}} h_u+1)\geq h(x,u)+\epsilon u.
\end{equation}
for all $i=1,\ldots,n$.

Although the right-hand side of the equation in this section depends on $u$, inequality \eqref{e7.18} played a similar role to that of \eqref{ne4.5}(see Definition \ref{d4.1}), it said that 0 is a $\mathcal{C}$-subsolution of \eqref{e7.10}. Hence we can use the same method as in Section 5 to get the following estimate:
\begin{equation}\label{e7.19}
\sup_\Omega|\sqrt{-1}\partial\bar{\partial}u|_\alpha\leq C\biggl(\sup_\Omega|\nabla u|_\alpha^2+1\biggl).
\end{equation}
The only difference is that when we estimate the term $F^{kk}g_{k\bar{k}1\bar{1}}$, we need to apply $\nabla_i$ and $\nabla_{\bar{i}}$ to the equation \eqref{e7.10} and set $i=1$, since $h\in \tilde{C}^\infty(\Omega \times\mathbb{R})$ and assuming $\lambda_1\gg \sup_\Omega|\nabla u|_\alpha^2+1$, we have
\begin{equation}\label{e7.20}
\begin{aligned}
F^{pq,rs}\nabla_1 g_{q\bar{p}}\nabla_{\bar{1}}g_{s\bar{r}}+F^{kk}\nabla_{\bar{1}}\nabla_1g_{k\bar{k}}&=h_{1\bar{1}}+h_{1u}u_{\bar{1}}+h_u u_{1\bar{1}}+h_{u\bar{1}}u_1+(h_{uu}+\epsilon)u_1u_{\bar{1}}\\
&\geq -C\mathcal{F}\lambda_1.
\end{aligned}
\end{equation}
By using \eqref{e7.20}, we can also get the inequality (3.18) in \cite{STW17}, so we can prove \eqref{e7.19}. The $C_1$ and high-order estimates can be proved by using the same method as in Section 5.  This completes the proof of the theorem.
\end{proof}

Next, we set up the continuity method to solve the equation \eqref{e7.10} as in Lemma \ref{t7.2}, we assume that
\begin{equation}
F(\alpha^{i\bar{p}}\chi_{j\bar{p}})=0\;\;\;{\rm and}\;\;\;\lambda[\alpha^{i\bar{p}}\chi_{j\bar{p}}]\in \Gamma_n.
\end{equation}

Define
\begin{equation}
\Psi:\tilde{C}^{k+\alpha}(\Omega)\times \mathbb{R}\rightarrow \tilde{C}^{k-2+\alpha}(\Omega)
\end{equation}
by
\begin{equation}
\Psi(u,t)=F(\alpha^{i\bar{p}}(\chi_{j\bar{p}}+u_{j\bar{p}}))-th(x,u)-\epsilon u.
\end{equation}
The fact that $\Psi$ maps $\tilde{C}^{k+\alpha}(\Omega)\times \mathbb{R}$ into $\tilde{C}^{k-2+\alpha}(M)$ can be easily verified using coordinate charts.

We are actually interested in the following open subset $U$ of $\tilde{C}^{k+\alpha}(\Omega)$:
\begin{equation}
\begin{aligned}
U=\{u\in \tilde{C}^{k+\alpha}(\Omega):(1/c)(\alpha_{j\bar{p}})<(\chi_{j\bar{p}}+u_{j\bar{p}})<c (\alpha_{j\bar{p}}), \;
&{\rm for \;some\; positive \; constant} \;c\}.
\end{aligned}
\end{equation}
To solve \eqref{e7.10}, we set $\mathcal{L}$ be the set $\{t\in [0,1]: {\rm \;there \;exists} \;u\in U \;{\rm such \;that} \; \Psi(u,t)=0\}$. It suffices to show that $1\in \mathcal{L}$. Indeed, we shall show that $\mathcal{L}=[0,1]$. First of all, observe that $\Psi(0,0)=0$, we have $0\in \mathcal{L}$. Then we shall show that $\mathcal{L}$ is both open and closed.

\emph{Openness}. To show that $\mathcal{L}$ is an open set, we shall use the usual implicit function theorem for Banach spaces. This amounts to showing that the Fr\'{e}chet derivative of $\Psi$ with respect to $u$, denoted as $\Psi_u$, has a bounded inverse.
The Fr\'{e}chet derivative $\Psi_u$ at the point $(u,t)$ is then an operator from $\tilde{C}^{k+\alpha}(\Omega)$ into $\tilde{C}^{k-2+\alpha}(\Omega)$ defined for any $g\in \tilde{C}^{k+\alpha}(\Omega)$,
\begin{equation}
(\Psi_u(u,t))(g)=F^{ij}(\alpha^{a\bar{p}}(\chi_{b\bar{p}}+u_{b\bar{p}}))\alpha^{i\bar{l}}g_{j\bar{l}}-th_u(x,u)g-\epsilon g.
\end{equation}
Therefore, we have to show that, for any $v\in \tilde{C}^{k-2+\alpha}(\Omega)$,
\begin{equation}\label{e7.26}
F^{ij}(\alpha^{a\bar{p}}(\chi_{b\bar{p}}+u_{b\bar{p}}))\alpha^{i\bar{l}}g_{j\bar{l}}-th_u(x,u)g-\epsilon g=v
\end{equation}
can be solved for $g\in \tilde{C}^{k+\alpha}(\Omega)$ and that $|g|_{k+\alpha}\leq c|v|_{k-2+\alpha}$ for some constant $c$ independent of $v$. By the $C^2$ estimates and condition \textbf{(a3)}, we have
\begin{equation}
\tilde{c}^{-1}\delta_{ij}<(F^{ij}(\alpha^{a\bar{p}}(\chi_{b\bar{p}}+u_{b\bar{p}})))<\tilde{c}\delta_{ij}
\end{equation}
for some constant $\tilde{c}>0$.

We first show that there is at most on function $g$ in $\tilde{C}^{k+\alpha}(\Omega)$ that satisfies equation \eqref{e7.26}. It is enough to check that $(\Psi_u(u,t))(g)=0$ and $g\in \tilde{C}^{k+\alpha}(\Omega)$ imply $g\equiv 0$. Assume $g\in \tilde{C}^{k+\alpha}(\Omega)$, $g$ is in particular bounded. We want to show that $g\leq 0$. The generalized maximum principle implies the existence of a sequence of points $\{x_i\}$ in $\Omega$ such that $\lim g(x_i)=\sup g$ and $\overline{\lim}(g_{p\bar{q}}(x_i))\leq 0$. Applying this to the equation $(\Psi_u(u,t))(g)=0$, we obtain $\sup g\leq 0$. Similarly, $\inf g\geq 0$, which implies that $g\equiv 0$.

Next, we prove the existence of $g$.

Let $\Omega_i$ be an exhaustion of $\Omega$ by compact subdomains. Assume $v\in \tilde{C}^{k-2+\alpha}(\Omega)$, and let $g^i$ be the unique solution to
\begin{equation}
\begin{aligned}
(\Psi_u(u,t))(g^i)=v \;\;\;{\rm on}\;\Omega_i,\\
g^i=0\;\;\;{\rm on}\;\partial\Omega_i.
\end{aligned}
\end{equation}
By applying the maximum principle to $\Omega_i$, we obtain that
\begin{equation}
\sup_{\Omega_i}|g^i|\leq \frac{\sup |v|}{\epsilon}.
\end{equation}
Interior Schauder estimates applied in coordinate charts as in Definition \ref{db2.1} immediately show that a subsequence of $g^i$ will converge to some $g\in \tilde{C}^{k+\alpha}(\Omega)$ which solves \eqref{e7.26} and that the estimate $|g|_{k+\alpha}\leq c|v|_{k-2+\alpha}$ is valid for some constant independent of $v$. This completes the proof of the openness of $\mathcal{L}$.

\emph{Closedness}. By Lemma \ref{t7.2}, we obtain the closedness of $\mathcal{L}$.

Then, by the above argument, we obtain:
\begin{lemma}\label{t7.3}
Assume that $F(\alpha^{i\bar{p}}\chi_{j\bar{p}})=0$, then the equation \eqref{e7.10} as in Lemma \ref{t7.2} exists a solution $u\in \tilde{C}^{k+\beta}(\Omega)$.
\end{lemma}

Finally, we obtain the following existence result (Theorem \ref{t7.4}) for the solution of $F(A)=h$.

\begin{proof}[Proof of Theorem \ref{t7.4}]
Replacing the right-hand side term $h(x,u)$ in equation \eqref{e1.10} by $h(x,u)+\epsilon_m u$, where $\{\epsilon_m\}$ is a sequence such that $\epsilon_m\in (0,1)$ and $\epsilon_m\rightarrow 0$, we obtain a sequence of the following equations,
\begin{equation}\label{e7.32}
F(\alpha^{i\bar{p}}(\chi_{j\bar{p}}+u_{j\bar{p}}))=h(x,u)+\epsilon_m u.
\end{equation}
By Lemma \ref{t7.2} and Lemma \ref{t7.3}, for each $m$, the equation \eqref{e7.32} exists a solution $u_m\in \tilde{C}^{k+\beta}(\Omega)$ and $|u_m|_{k+\beta}<c$ for some constant $c$ independent of $m$. Let $\{\Omega_i\}$ be an exhaustion of $\Omega$ by compact subdomains. By the Arzel\`{a}-Ascoli Theorem, on $\Omega_1$ we can take a subsequence $\{u_{m_k}\}$ of $\{u_m\}$ such that $u_{m_k}\rightarrow u^1$ for some $u^1\in \tilde{C}^{k+\beta}(\Omega_1)$. And on $\Omega_2$, we can also take a subsequence of $\{u_{m_k}\}$ converging to some $u^2\in \tilde{C}^{k+\beta}(\Omega_2)$. Note that $u^2|_{\Omega_1}=u^1$. Continuing in this way, we obtain a function sequence $u^i\in \tilde{C}^{k+\beta}(\Omega_i)$ such that $u^i|_{\Omega_{i-1}}=u^{i-1}$ and $|u_i|_{k+\beta}<c$ for some constant $c$ independent of $i$. Define $u\in \tilde{C}^{k+\beta}(\Omega)$ by $u|_{\Omega_i}=u^i$ for each $i$, we can verify that $u$ satisfies equation \eqref{e1.10}.
\end{proof}

If, in addition, assuming that $h_u(x,u)$ is strictly positive, we can obtain the following:

\begin{proof}[Proof of Theorem \ref{t7.5}]
The proof is similar to that of Lemma \ref{t7.3}. We only need to use Theorem \ref{t6.4} instead of Theorem \ref{t6.1} to obtain the $C^0$ estimates for \eqref{e7.33}, the rest of the proof is the same as that of Lemma \ref{t7.3}.
\end{proof}

\section{Complex Monge-Amp\`{e}re equations}

In this section, we consider the complex Monge-Amp\`{e}re equation and $(n-1)$ Monge-Amp\`{e}re equation on noncompact manifolds (which can be seen as the special cases of \eqref{e1.1}) and some geometric applications.

When
\begin{equation}\label{e8.1}
f(\lambda_1,\ldots,\lambda_n)=\log\lambda_1\ldots\lambda_n,
\end{equation}
where $f$ is defined on $\Gamma_n$, the equation \eqref{e1.1} is the Monge-Amp\`{e}re equation. It is straightforward to check that $f$ satisfies the structure conditions for equation \eqref{e1.1}. \textbf{(a1)}, \textbf{(a2)} and \textbf{(a5)} are obvious. Indeed, $f$ converges to $-\infty$ on the boundary $\partial \Gamma_n$, so \textbf{(a3)} is satisfied, and for \textbf{(a4)} it is enough to note that $f(t\mu)=f(\lambda)+n\log t$, which converges to $\infty$ as $t\rightarrow \infty$.

In addition, if $\lambda\in \Gamma_n$, then we also have
\begin{equation}
\lim_{t\rightarrow\infty}f(\lambda+t\textbf{e}_i)=\infty
\end{equation}
for all $i$. This means that for a function $u$ to be a $\mathcal{C}$-subsolution for this equation, the only requirement is that $\lambda[\alpha^{i\bar{p}}\chi_{j\bar{p}}]\in \Gamma_n$ (see Page 1).

Let \eqref{e1.10} be a Monge-Amp\`{e}re equation, then by Theorem \ref{t7.4}, we obtain the following:
\begin{corollary}\label{c77.1}
Let $\Omega$ be a strictly pseudoconvex domain in a complete K\"{a}hler manifold $(M,\hat{h}_{i\bar{j}})$ as in Proposition \ref{pb2.3} with $l\geq3$. Define $g=-\log(-\varphi)$ and $\alpha_{i\bar{j}}=g_{i\bar{j}}+\hat{h}_{i\bar{j}}$, where $\varphi$ is the defining function of $\Omega$. Given $h\in \tilde{C}^{l-1+\beta}(\Omega)$, where $\beta\in (0,1)$ is some uniform constant. Consider the following equation on $\Omega$:
\begin{equation}\label{e8.2}
\det(\alpha_{i\bar{j}}+u_{i\bar{j}})=e^h\det(\alpha_{i\bar{j}})
\end{equation}
Suppose that $|h|=O(\varphi^2 \underset{\min}{\lambda}(\hat{h}^{a\bar{j}}\varphi_{b\bar{j}}))$, then the equation \eqref{e8.2} exists a solution $u\in \tilde{C}^{l+1+\beta}(\Omega)$ such that $(\alpha_{i\bar{j}}+u_{i\bar{j}})$ is uniformly equivalent to $\alpha$.
\end{corollary}

The above corollary can be used to find K\"{a}hler metrics with some prescribed volume forms on strictly pseudoconvex domains, and we can prove Theorem \ref{t1.7}.
\begin{proof}[Proof of Theorem \ref{t1.7}] Since $|\partial\bar{\partial}\varphi|_{h}>c>0$, we obtain that $(\Omega,g_{i\bar{j}})$ has bounded geometry of infinity order. Note that
\begin{equation}
\varphi^{2(1-n)}=O(\varphi^2) |(\sqrt{-1})^n\det(g_{i\bar{j}})dz^1 \wedge d\bar{z}^1\cdots \wedge dz^n\wedge d\bar{z}^n|_{h},
\end{equation}
then by \eqref{e1.12}, we have
\begin{equation}
\frac{|V(x)-\det(g_{i\bar{j}})|}{\det(g_{i\bar{j}})}=O(\varphi^2).
\end{equation}
Finally,
\begin{equation}
V(x)=e^{\tilde{h}} \det(g_{i\bar{j}}),
\end{equation}
where $\tilde{h}\in \tilde{C}^\infty(\Omega)$ and $|\tilde{h}|=O(\varphi^2)$. Consider the following equation
\begin{equation}\label{e777.7}
\det(g_{i\bar{j}}+u_{i\bar{j}})=V(x)=e^{\tilde{h}} \det(g_{i\bar{j}}),
\end{equation}
taking $g_{i\bar{j}}$ instead of $\alpha_{i\bar{j}}$ in \eqref{e8.2} of Corollary \ref{c77.1}, note that $\underset{\min}{\lambda}(h^{a\bar{j}}\varphi_{b\bar{j}})>c'>0$, we obtain that \eqref{e777.7} has a solution $u\in \tilde{C}^{\infty}(\Omega)$ such that $(g_{i\bar{j}}+u_{i\bar{j}})$ is uniformly equivalent to $\alpha$. This completes the proof of the theorem.
\end{proof}

Let \eqref{e7.33} be a Monge-Amp\`{e}re equation (see \eqref{e8.1}), then by Theorem \ref{t7.5}, we obtain the following:
\begin{corollary}\label{c8.2}
Let $(\Omega,\alpha)$ be a complete Hermitian manifold with bounded geometry of order $k-1$, where $k\geq 4$.
Then for any $K\geq0$ and $h\in \tilde{C}^{k-2+\beta}(\Omega)$, $\beta\in (0,1)$. Consider the following equation on $\Omega$:
\begin{equation}\label{e8.3}
\det(\alpha_{i\bar{j}}+u_{i\bar{j}})=e^{Ku}e^h\det(\alpha_{i\bar{j}}),
\end{equation}
Then the equation \eqref{e8.3} exists a solution $u\in \tilde{C}^{k+\beta}(\Omega)$ such that $\chi+\sqrt{-1}\partial\bar{\partial}u$ is uniformly equivalent to $\alpha$. Moreover, if
\begin{equation}\label{e8.4}
{\rm Ric}(\alpha)=-K'\alpha+\partial\bar{\partial}h'
\end{equation}
for some $K'>0$ and $h'\in \tilde{C}^{k-2+\beta}(\Omega)$. Then, there exists a complete Chern-Einstein metric on $\Omega$.
\end{corollary}
\begin{proof}
Let $u\in \tilde{C}^{k+\beta}(\Omega)$ be a solution of the following equation on $(\Omega,\alpha)$:
\begin{equation}\label{e8.5}
\det(\alpha_{i\bar{j}}+u_{i\bar{j}})=e^{K'u}e^{h'}\det(\alpha_{i\bar{j}}).
\end{equation}
By \eqref{e8.4} and \eqref{e8.5}, we have
\begin{equation}
\begin{aligned}
{\rm Ric}(\alpha_{p\bar{q}}+u_{p\bar{q}})_{i\bar{j}}&=-\frac{\partial^2}{\partial z^i\partial\bar{z}^j}[\log \det(\alpha_{p\bar{q}}+u_{p\bar{q}})]\\
&=-\frac{\partial^2}{\partial z^i\partial\bar{z}^j}[K'u+h'+\log \det(\alpha_{p\bar{q}})]\\
&=-K'u_{i\bar{j}}-h'_{i\bar{j}}+{\rm Ric}(\alpha)_{i\bar{j}}\\
&=-K'(\alpha_{i\bar{j}}+u_{i\bar{j}}).
\end{aligned}
\end{equation}
Hence $g_{i\bar{j}}:=\alpha_{i\bar{j}}+u_{i\bar{j}}$ is a complete Chern-Einstein metric (${\rm Ric}(g)=-K'g$).
\end{proof}
Corollary \ref{c8.2} can be regarded as a generalization of Theorem 4.4 in \cite{CY80} to Hermitian manifolds. When $(\Omega,\alpha)$ is a K\"{a}hler manifold, we refer readers to \cite{CY80} to see some types of K\"{a}hler manifolds that satisfy the condition \eqref{e8.4}, and then such manifolds exist complete K\"{a}hler-Einstein metrics with negative
scalar curvature.

Indeed, Corollary \ref{c8.2} can be used to find K\"{a}hler-Einstein metrics on some submanifolds of a complete K\"{a}hler manifold whose first Chern class is negative. The following functions $f$, $\psi$ and $\mathfrak{F}$ can be found in \cite{LT20}:

Let $\kappa\in (0,1)$, $f:[0,1)\rightarrow [0,\infty)$ be the function:
\begin{equation}
\label{e5.2}
f(s)=\left\{ \begin{aligned}
&0,\;\;\;\;\;\;\;\;\;\;\;\;\;\;\;\;\;\;\;\;\;\;\;\;\;\;\;\;\;\;\;\;\;\;\;\;\;\;\;s\in[0,1-\kappa];\\
&-\log\biggl[1-\biggl(\frac{s-1+\kappa}{\kappa}\biggl)^2\biggl],s\in(1-\kappa,1).
\end{aligned} \right.
\end{equation}
Let $\psi\geq 0$ be a smooth function on $\mathbb{R}^+$ such that 
\begin{equation}
\label{e5.3}
\psi(s)=\left\{ \begin{aligned}
&0,\;\;\;s\in[0,1-\kappa+\kappa^2];\\
&1,\;\;\;s\in(1-\kappa+2\kappa^2,1).
\end{aligned} \right.
\end{equation}
and $\frac{2}{\kappa^2}\geq \psi'\geq 0$. Define
\[\mathfrak{F}:=\int_0^s\psi(\tau)f'(\tau)d\tau\]
By \cite{LT20}, we have
\begin{lemma}\label{l77.3}Suppose $0<\kappa<\frac{1}{8}$. Then the function $\mathfrak{F}\geq 0$ defined above is smooth and satisfies the following:\\
\hspace*{0.5cm}\emph{(i)} $\mathfrak{F}=0$ for $s\in[0,1-\kappa+\kappa^2]$.\\
\hspace*{0.4cm}\emph{(ii)} $\mathfrak{F}'\geq 0$ and for any $k\geq 1$, $\exp(-k\mathfrak{F})\mathfrak{F}^{(k)}$ is uniformly bounded.\\
\end{lemma}

Then, we can prove Theorem \ref{t1.8}, we rewrite it here:

\begin{theorem}
Let $(M,\alpha)$ be a complete K\"{a}hler manifold with negative Ricci curvature bounded from below by a negative
constant. Then, for any compact subset $K\subset M$, there exists a K\"{a}hler-Einstein metric on $K$. Moreover, this metric can be extended to some open set $N\subset M$ and forms a complete Hermitian metric on $N$.
\end{theorem}
\begin{proof}
Since that ${\rm Ric}(\alpha)$ is bounded from below by a negative constant, by Shi \cite{shi97} (see also Lemma 3.2 of \cite{JY25}),
there exists a smooth real function $\rho$ which is uniformly equivalent to the distance function from a fixed point. For any $\rho_0>0$, let $U_{\rho_0}$ be the component of $\{x|\rho(x)<\rho_0\}$ that contains the fixed point as above. Hence $U_{\rho_0}$ will exhaust $M$ as $\rho_0\rightarrow \infty$. For $\rho_0>>1$, we define $F(x)=\mathfrak{F}(\rho(x)/\rho_0)$. Let $h_0=-e^{2F}{\rm Ric}(\alpha)$. We find $(U_{\rho_0},h_0)$ is a complete Hermitian metric with bounded geometry of infinite order (see \cite{LT20}) and $h_0=-{\rm Ric}(\alpha)$ on $\{\rho(x)<(1-\kappa+\kappa^2)\rho_0\}$. If we choose $\rho_0$ large enough, we can assume the $K\subset \{\rho(x)<(1-\kappa+\kappa^2)\rho_0\}$. Consider the following equation on $U_{\rho_0}$:
\begin{equation}\label{e77.9}
\det((h_0)_{i\bar{j}}+u_{i\bar{j}})=e^{u}e^h\det((h_0)_{i\bar{j}}),
\end{equation}
where
\begin{equation}\label{e77.10}
h:=\log \frac{\det \alpha}{\det (-{\rm Ric}(\alpha))}.
\end{equation}

Since $\overline{U_{\rho_0}}$ is compact, the K\"{a}hler metric $-{\rm Ric}(\alpha)$ has bounded geometry of infinity order on $U_{\rho_0}$, and $h\in\tilde{C}^\infty(U_{\rho_0})$ with respect to $-{\rm Ric}(\alpha)$. By Lemma \ref{l77.3}, we have $h_0\geq -{\rm Ric}(\alpha)$, hence we also have $h\in\tilde{C}^\infty(U_{\rho_0})$ with respect to $h_0$. Then by Corollary \ref{c8.2}, the equation \eqref{e77.9} exists a solution $u\in \tilde{C}^\infty(U_{\rho_0})$ such that $((h_0)_{i\bar{j}}+u_{i\bar{j}})$ is uniformly equivalent to $((h_0)_{i\bar{j}})$. By \eqref{e77.9} and \eqref{e77.10}, we obtain
\begin{equation}
\begin{aligned}
{\rm Ric}((h_0)_{p\bar{q}}+u_{p\bar{q}})_{i\bar{j}}&=-\frac{\partial^2}{\partial z^i\partial\bar{z}^j}[\log \det((h_0)_{p\bar{q}}+u_{p\bar{q}})]\\
&=-\frac{\partial^2}{\partial z^i\partial\bar{z}^j}[u+\log \frac{\det \alpha}{\det (-{\rm Ric}(\alpha))}+\log \det((h_0)_{p\bar{q}})]\\
&=-((h_0)_{i\bar{j}}+u_{i\bar{j}}),
\end{aligned}
\end{equation}
on $\{\rho(x)<(1-\kappa+\kappa^2)\rho_0\}$. Finally, $((h_0)_{i\bar{j}}+u_{i\bar{j}})$ is a K\"{a}hler-Einstein metric on $K$ and it is a complete Hermitian metric on $N:=U_{\rho_0}$.
\end{proof}

Next, we consider the form-type $(n-1)$-Monge-Amp\`{e}re equation on $M$, which can be seen as another example of \eqref{e1.1}. Some basic properties about $(n-1,n-1)$ forms can be found in Appendix A.

\begin{equation}\label{eb2.10}
\left\{ \begin{aligned}
\det\Big[&\omega_0^{n-1}+\sqrt{-1}\partial\overline{\partial}v\wedge \omega^{n-2}\Big]=e^h\det\omega_0^{n-1},\\
&\omega_0^{n-1}+\sqrt{-1}\partial\overline{\partial}v\wedge \omega^{n-2}>0\;\;\;{\rm on}\; M,
\end{aligned} \right.
\end{equation}
where $\omega$ is a K\"{a}hler form as in \eqref{eb2.6} and $\omega_0$ is the form of a balanced metric on $M$.

Then by \eqref{eb2.7}-\eqref{eb2.12}, we have
\begin{equation}\label{eb2.13}
\begin{aligned}
&\frac{\det(*(\omega_0^{n-1}+\sqrt{-1}\partial\overline{\partial}v\wedge \omega^{n-2}))}{\det((n-1)!\omega)}\\
=&\frac{\det(*(\omega_0^{n-1}+\sqrt{-1}\partial\overline{\partial}v\wedge \omega^{n-2}))}{\det(*(\omega^{n-1}))}\\
=&\frac{\det(\omega_0^{n-1}+\sqrt{-1}\partial\overline{\partial}v\wedge \omega^{n-2})}{\det(\omega^{n-1})}\\
=&e^F\frac{\det \omega_0^{n-1}}{\det \omega^{n-1}},
\end{aligned}
\end{equation}
where $*$ is the Hodge star operator with respect to $\omega$, and
\begin{equation}\label{eb2.14}
\begin{aligned}
*(\omega_0^{n-1}+\sqrt{-1}\partial\overline{\partial}v\wedge \omega^{n-2})&=*\omega_0^{n-1}+*(\sqrt{-1}\partial\overline{\partial}v\wedge \omega^{n-2})\\
&=(n-1)!\biggl(\omega_h+\frac{1}{n-1}((\Delta v)\omega-\sqrt{-1}\partial\bar{\partial}v)\biggl),
\end{aligned}
\end{equation}
where $\Delta$ is the Laplacian operator associated to the metric $\omega$. $\omega_h$ is the Hermitian metric defined by
\begin{equation}
\omega_h=\frac{1}{(n-1)!}*\omega_0^{n-1}.
\end{equation}
By \eqref{eb2.13} and \eqref{eb2.14}, we see that \eqref{eb2.10} is equivalent to
\begin{equation}\label{eb2.16}
\log(\mu_1\cdots\mu_n)=h+\log \frac{\det \omega_0^{n-1}}{\det \omega^{n-1}},
\end{equation}
where $\mu_i$ are the eigenvalues of $\omega^{i\bar{p}}\tilde{g}_{j\bar{p}}$, for $\tilde{g}$ given by
\begin{equation}\label{8.23}
\tilde{g}_{i\bar{j}}=(\omega_h)_{i\bar{j}}+\frac{1}{n-1}((\Delta v)\omega_{i\bar{j}}-v_{i\bar{j}}).
\end{equation}

Now we will show that \eqref{8.23} is a special case of \eqref{e1.1} that satisfies conditions \textbf{(a1)}-\textbf{(a5)}.

We assume that our operator $F$ (see \eqref{e1.1}) has the special form $F(M)=\widetilde{F}(P(M))$, where
\begin{equation}
P(M)=(n-1)^{-1}({\rm Tr}(M)I-M),
\end{equation}
(here writing ${\rm Tr}(M)$ for the trace of matrix $M$), and
\begin{equation}
\widetilde{F}(B)=\log(\mu_1,\ldots,\mu_n),
\end{equation}
where $\mu_1,\ldots,\mu_n$ are the eigenvalues of $B$. In terms of eigenvalues, this means that
\begin{equation}
f(\lambda_1,\ldots,\lambda_n)=(\log\circ P)(\lambda_1,\ldots,\lambda_n),
\end{equation}
where we are writing $P$ for the map $\mathbb{R}^n\rightarrow \mathbb{R}^n$ induced on diagonal matrices by the matrix map $P$ above. Explicitly, writing $\mu=P(\lambda)$ for the corresponding $n$-tuples $\lambda,\mu\in \mathbb{R}^n$, we have
\begin{equation}
f(\lambda_1,\ldots,\lambda_n)=\log(\mu_1,\ldots,\mu_n),\;\;\;\;{\rm for}\;\mu_k=\frac{1}{n-1}\sum_{i\neq k}\lambda_i.
\end{equation}

Define the cone $\Gamma_n\subset \mathbb{R}^n$ by $\Gamma_n=P^{-1}(\widetilde{\Gamma_n})$. Observe that $P$ maps $\Gamma_n$ into $\Gamma_n$. It is then easy to see that the function $f:\Gamma\rightarrow \mathbb{R}$ satisfies exactly the same conditions as the function $\log(\mu_1,\ldots,\mu_n)$ (condition \textbf{(a1)}-\textbf{(a5)}).

By the above argument, the equation \eqref{eb2.16} can be rewritten in the following form:
\begin{equation}
F(A)=\log\det P(A)=h(x,u).
\end{equation}
Hence we can use Theorem \ref{t7.4} and Theorem \ref{t7.5} to solve \eqref{eb2.16}.

\section{Applications in affine geometry}

In this section, we solve a class of fully non-linear equations and construct the Hesse-Einstein metrics on certain noncompact affine manifolds. The key technology of the section is to lift the equations on an affine manifold to it's tangent bundle (which can be viewed as a Hermitian manifold), and use the a priori estimates for \eqref{e1.1} to derive the a priori estimates for some equations on affine manifolds.

Some basic results and preliminaries about affine manifolds and Hessian manifolds can be found in Appendix B. In this section, $\nabla$ denotes the flat connection. Now we introduce the relationship between affine structures and complex structures.
Let $(M,\nabla)$ be a flat manifold and $TM$ be the tangent bundle of $M$ with projection $\pi:TM\rightarrow M$. For an affine coordinate system $\{x^1,\ldots,x^n\}$ on $M$, we define
\begin{equation}\label{e99.2}
z^j=\xi^j+\sqrt{-1}\xi^{n+j}
\end{equation}
where $\xi^i=x^i\circ \pi$ and $\xi^{n+i}=dx^i$. Then $n$-tuples of functions given by $\{z^1,\ldots,z^n\}$ yield holomorphic coordinate systems on $TM$. For any smooth function $u$ on $M$, for any point $p\in TM$, we have
\begin{equation}
\frac{\partial (u\circ\pi)}{\partial z_i}(p)=\frac{\partial (u\circ\pi)}{\partial \bar{z_i}}(p)=\frac{1}{2}\frac{\partial u}{\partial x_i}(\pi(p)).
\end{equation}

We denote by $J_\nabla$ the complex structure tensor of the complex manifold $TM$. For a Riemannian metric $g$ on $M$ we put
\begin{equation}
g^T=\sum_{i,j=1}^{n}(g_{ij}\circ \pi)dz^id\overline{z}^j.
\end{equation}
Then $g^T$ is a Hermitian metric on the complex manifold $(TM,J_\nabla)$, and for any $(0,2)$-tensor $\chi$, $\chi^T$ can be defined similarly. Let $\beta (g)$ is a form associated to the Riemannian metric $g$, which in local affine coordinates is given by
\[
{{\beta }_{ij}}(g)=-{{\partial }_{i}}{{\partial }_{j}}\log \det [g(t)]=-2\kappa_{ij}.
\]
If $\beta(g)=\lambda g$ for some constant $\lambda$, then $g$ is called the Hesse-Einstein metric.

Let $(M, \nabla, g)$ be a complete affine manifold of dimension $n$. Fix a $(0,2)$-tensor $\chi$ on $(M,g)$, for any $C^2$ function $u:M\rightarrow \mathbb{R}$ we have a new $(0,2)$-tensor $\hat{g}=\chi+\nabla du$, and we can define $A_j^i=g^{ip}\hat{g}_{jp}$. We consider the equation for $u$ as follows:
\begin{equation}\label{e99.3}
F(A)=h(x,u),
\end{equation}
for a given function $h$ on $M$, where
\begin{equation}
F(A)=f(\lambda_1,\ldots,\lambda_n)
\end{equation}
is a smooth symmetric function of the eigenvalues of $A$ and satisfies the structure conditions \textbf{(a1)}-\textbf{(a5)}. Then if $u$ is a solution of
\eqref{e99.3}, by \eqref{e99.2}, we obtain that $u\circ\pi$ satisfies the following equation:
\begin{equation}\label{e99.5}
F((g^T)^{i\bar{p}}(\chi^T_{j\bar{p}}+4(u\circ\pi)_{j\bar{p}}))=h(x,u\circ\pi)
\end{equation}
on $(TM,g^T)$. This means that we can use the a priori estimates for \eqref{e99.5} to derive the a priori estimates for \eqref{e99.3}.

Note that on affine/Hessian manifolds, we can similarly define bounded geometry, $|\cdot|_{k+\alpha}$, $\tilde{C}^{k+\alpha}(M)$, etc. analogous to Definition \ref{db2.1}. And a Hessian manifold $M$ is said to have a bounded global Hessian potential if and only if there exists a bounded $C^\infty$ function $\varphi:M\rightarrow \mathbb{R}$ such that $\nabla d\varphi>0$ (similar to Definition \ref{d1.1}). Consequently, the existence results for solutions of \eqref{e99.3} analogous to Theorems \ref{t7.4} and \ref{t7.5} can be obtained. We state the following corollary, which is not presented in its most general form.

\begin{corollary}\label{c9.1}
Let $(\Omega,g)$ be a complete affine  Riemannian manifold with bounded geometry of order $k-1$, where $k\geq 4$.
Then for any $K\geq0$ and $h\in \tilde{C}^{k-2+\beta}(\Omega)$, $\beta\in (0,1)$. Consider the following equation on $\Omega$:
\begin{equation}\label{e9.1}
\det(g_{ij}+u_{ij})=e^{Ku}e^h\det(g_{ij}),
\end{equation}
Then the equation \eqref{e9.1} exists a solution $u\in \tilde{C}^{k+\beta}(\Omega)$ such that $(g_{ij}+u_{ij})$ is uniformly equivalent to $(g_{ij})$.
\end{corollary}
\begin{proof}
First, we consider the a priori estimates for \eqref{e8.1}, if $u\in \tilde{C}^{k+\beta}(\Omega)$ satisfies \eqref{e9.1}, then we have
\begin{equation}\label{e9.2}
\det(g_{ij}^T+(4u\circ\pi)_{i\bar{j}})=e^{Ku\circ\pi}e^h\det(g_{ij}^T)
\end{equation}
on $T\Omega$. We can easily verify that \eqref{e9.2} satisfies the conditions in Theorem \ref{t7.5} (see also Corollary \ref{c8.2}), and thus obtain the a priori estimates for $u\circ\pi$ up to $C^{k+\beta}$. Consequently, we also obtain the a priori estimates for $u$ up to $C^{k+\beta}$. We may then set up the continuity method to solve equation \eqref{e9.1}; the procedure is similar to that in Section 6, so we omit the proof here.
\end{proof}

Then, we want to construct the Hesse-Einstein metrics on complete Hessian manifolds whose the first affine Chern class is negative. We rewrite Theorem \ref{t1.9} here:

\begin{theorem}
Let $(\Omega,g)$ be a complete Hessian manifold with bounded geometry of infinity order. Suppose that the second Koszul form of $(\Omega,g)$ is bounded from below by a positive constant, then $M$ admits a complete Hesse-Einstein metric.
\end{theorem}
\begin{proof}
Since that the second Koszul form $M$ is bounded from below by a positive constant, $-\beta (g)=2\kappa(g)$ can be seen as a Hessian metric with bounded geometry of infinity order on $M$. Consider the following equation
\begin{equation}\label{e9.3}
\det(-\beta_{ij}(g)+u_{ij})=e^{u}\exp^{\{\log\frac{\det(g)}{\det(-\beta (g))}\}}\det(-\beta_{ij}(g)),
\end{equation}
we can verify that the above equation satisfies the conditions in Corollary \ref{c9.1}, then \eqref{e9.3} exists a solution $u\in \tilde{C}^\infty(\Omega)$, and we have
\begin{equation}
\begin{aligned}
\beta_{ij}(-\beta(g)+\nabla du)&=-\frac{\partial^2}{\partial x^i\partial x^j}[\log \det(-\beta_{pq}(g)+u_{pq})]\\
&=-\frac{\partial^2}{\partial x^i\partial x^j}[u+\log\frac{\det(g)}{\det(-\beta (g))}+\log \det(-\beta_{pq}(g))]\\
&=-(-\beta_{ij}(g)+u_{ij}).
\end{aligned}
\end{equation}
Finally, $(-\beta(g)+\nabla du)$ is the complete Hesse-Einstein metric as we need.
\end{proof}

\appendix

\section{$(n-1,n-1)$ forms}
Let $M$ be a complex manifold of dimension $n>2$. There is a bijection from the space of positive definite $(1,1)$ forms to positive definite $(n-1,n-1)$ forms, give by
\begin{equation}\label{eb2.1}
\omega\mapsto\omega^{n-1}.
\end{equation}

Here are some conventions: For an $(n-1,n-1)$-form $\Theta$, we denote
\begin{equation}\label{eb2.2}
\begin{aligned}
\Theta=&(\sqrt{-1})^{n-1}(n-1)!\\
&\times \sum_{p,q}s(p,q)\Theta_{p\bar{q}}dz^1 \wedge d\bar{z}^1\cdots \wedge \widehat{dz^p}\wedge d\bar{z}^p \wedge \cdots \wedge dz^q\wedge \widehat{d\bar{z}^q} \cdots \wedge dz^n\wedge d\bar{z}^n,
\end{aligned}
\end{equation}
where
\begin{equation}
s(p,q)=\left\{ \begin{aligned}
-1,\;\;\;\;{\rm if} \;p>q;\\
1,\;\;\;\;\;{\rm if}\;p\leq q.
\end{aligned} \right.
\end{equation}
Here we introduce the sign function $s$ such that,
\begin{equation}
\begin{aligned}
dz^p&\wedge d\bar{z}^q \wedge s(p,q) \wedge dz^1 \wedge d\bar{z}^1\cdots \wedge \widehat{dz^p}\wedge d\bar{z}^p \wedge \cdots \wedge dz^q\wedge \widehat{d\bar{z}^q} \cdots \wedge dz^n\wedge d\bar{z}^n\\
&=dz^1 \wedge d\bar{z}^1\cdots \wedge dz^n\wedge d\bar{z}^n,\;\;\;{\rm for \;all}\;\;1\leq p,q\leq n.
\end{aligned}
\end{equation}

We denote
\begin{equation}
\det \Theta=\det(\Theta_{p\bar{q}}).
\end{equation}
If $\det \Theta\neq 0$, we denote the transposed inverse of $(\Theta_{p\bar{q}})$ by $(\Theta^{p\bar{q}})$, i.e.,
\[\sum_l \Theta_{i\bar{l}}\Theta^{j\bar{l}}=\delta_{ij}.\]
We note that, for a positive $(1,1)$-form $\omega$ given by
\begin{equation}\label{eb2.6}
\omega=\sqrt{-1}\sum_{i,j=1}^n g_{i\bar{j}}dz_i\wedge d\bar{z}_j,
\end{equation}
we obtain
\begin{equation}\label{eb2.7}
\omega^n=(\sqrt{-1})^n n!\det(g_{i\bar{j}})dz^1 \wedge d\bar{z}^1\cdots \wedge dz^n\wedge d\bar{z}^n,
\end{equation}
and we have
\begin{equation}\label{eb2.8}
(\omega^{n-1})_{i\bar{j}}=\det(g_{i\bar{j}})g^{i\bar{j}}.
\end{equation}
It follows that
\begin{equation}\label{eb2.9}
\det(\omega^{n-1})=\det(g_{i\bar{j}})^{n-1}.
\end{equation}

Write $*$ for the Hodge star operator with respect to $\omega$. The Hodge star operator $*$ takes any $(n-1,n-1)$ form to a $(1,1)$ form, and vice versa. This acts on real $(n-1,n-1)$-forms as follows. Consider a real $(n-1,n-1)$ given by \eqref{eb2.2}, if we are computing at a point in coordinates so that $\omega_{i\bar{j}}=\delta_{ij}$, then
\begin{equation}\label{eb2.11}
*\Theta=\sqrt{-1}(n-1)!\sum_{p,q}\Theta_{p\bar{q}}dz^q\wedge d\bar{z}^p,
\end{equation}
and
\begin{equation}\label{eb2.12}
**\Theta=\Theta.
\end{equation}

\section{Affine/Hessian manifolds}

\begin{definition} An affine manifold $(M,\nabla)$ is a differentiable manifold endowed with a flat connection $\nabla$.
\end{definition}
The subsequent statement is an immediate consequence of the affine manifold definition.
\begin{proposition}
\label{prop1}
\

\emph{(1)}
\begin{minipage}[t]{0.92\linewidth}
Suppose that $M$ admits a flat connection $\nabla$. Then there exist local coordinate systems on $M$ such that $\nabla_{\partial /\partial x^i} \partial /\partial x^j=0$. The changes between such coordinate systems are affine transformations.
\end{minipage}

\emph{(2)}
\begin{minipage}[t]{0.92\linewidth}
Conversely, if $M$ admits local coordinate systems such that the changes of the local coordinate systems are affine transformations, then there exists a flat connection $\nabla$ satisfying $\nabla_{\partial /\partial x^i} \partial /\partial x^j=0$ for all such local coordinate systems.
\end{minipage}
\end{proposition}

Recently, Lee-Topping \cite{LT25} solved the Hamilton's pinching conjecture, they prove that a class of Riemannian manifolds is flat, see the following:

\begin{theorem}
Suppose $(M^3,g_0)$ is a complete (connected) noncompact three-dimensional Riemannian manifold with ${\rm Ric}\geq \varepsilon \mathcal{R}\geq 0$ for some $\varepsilon>0$. Then $(M^3,g_0)$ is flat.
\end{theorem}

Let $(M, \nabla, g)$ be an affine manifold with a Riemannian metric $g$, where $\hat{\nabla}$ denotes the Levi-Civita connection of
$(M,g)$. Defining $\gamma=\hat{\nabla}-\nabla$. Since both connections $\nabla$ and $\hat{\nabla}$ are torsion-free, we obtain
\[\gamma_X Y=\gamma_Y X\]
Furthermore, in affine coordinate systems, the components $\gamma_{\;jk}^i$ of $\gamma$
coincide with the Christoffel symbols $\Gamma_{\;jk}^i$ associated with the Levi-Civita connection $\hat{\nabla}$.

For the pair $(g,\nabla)$, the Hessian curvature tensor is given by $Q=\nabla \gamma$.
A Riemannian metric $g$ on a affine manifold $(M,\nabla)$ is called Hessian if and only if it admits local expression through $g=\nabla d\varphi$. An affine manifold endowed with a Hessian metric is called a Hessian manifold.

Given a flat connection $\nabla$, a local coordinate system $\{x^1,\ldots,x^n\}$ satisfying $\nabla_{\frac{\partial}{\partial x^i}} \frac{\partial}{\partial x^j}=0$ is called an affine coordinate system with respect to $\nabla$.
Let $(M,\nabla,g)$ be a Hessian manifold and its metric $g$ admits local representation
\[g_{ij}=\frac{\partial^2 \phi}{\partial x^i \partial x^j}\]
where $\{x^1,\ldots,x^n\}$ is an affine coordinate system with respect to $\nabla$. The next two results are established in \cite{r11}.

\begin{proposition}\label{p2.3}
Let $(M,\nabla)$ be an affine manifold and $g$ a Riemannian metric on $M$. Then the following are equivalent:

\emph{(1)} $g$ is a Hessian metric

\emph{(2)} $(\nabla_X g)(Y,Z)=(\nabla_Y g)(X,Z)$

\emph{(3)} $\dfrac{\partial g_{ij}}{\partial x^k}=\dfrac{\partial g_{kj}}{\partial x^i}$

\emph{(4)} $g(\gamma_X Y,Z)=g(Y,\gamma_X Z)$

\emph{(5)} $\gamma_{ijk}=\gamma_{jik}$

\end{proposition}

\begin{proposition}\label{p2.4}
 Let $\hat{R}$ be the Riemannian curvature of a Hessian metric $g=\nabla d\phi$ and $Q=\nabla \gamma$ be the Hessian curvature tensor for $(g,\nabla)$. Then

\emph{(1)} $Q_{ijkl}=\dfrac{1}{2}\dfrac{\partial^4 \phi}{\partial x^i \partial x^j \partial x^k \partial x^l}-\dfrac{1}{2}g^{pq}\dfrac{\partial^3 \phi}{\partial x^i \partial x^k \partial x^p}\dfrac{\partial^3 \phi}{\partial x^j \partial x^l \partial x^q}$

\emph{(2)} $\hat{R}(X,Y)=-[\gamma_X,\gamma_Y],\;\;\;\hat{R}^i_{\;jkl}=\gamma^i_{\;lm}\gamma^m_{\;jk}-\gamma^i_{\;km}\gamma^m_{\;jl}.$

\emph{(3)}
$\hat{R}_{ijkl}=\dfrac{1}{2}(Q_{ijkl}-Q_{jikl})$

\end{proposition}

\begin{definition}
\label{Kos}
Let $(M,\nabla,g)$ be a Hessian manifold and
$v$ the volume element of $g$. The first Koszul form $\alpha$ and the second Koszul form $\kappa$ for $(\nabla,g)$ are defined by
\[\nabla_X v=\alpha(X)v\]
\[\kappa=\nabla \alpha\]
\end{definition}
It follows that
\[\alpha(X)={\rm Tr} {\gamma_X}\]
and
\[\alpha_i=\frac{1}{2}\frac{\partial \log \det[g_{pq}]}{\partial x^i}=\gamma_{\;ki}^k\]
\[\kappa_{ij}=\frac{\partial \alpha_i}{\partial x^j}=\frac{1}{2}\frac{\partial \log \det[g_{pq}]}{\partial x^i \partial x^j}\]
locally.


\end{document}